\documentclass[a4paper,12pt]{article}
\usepackage{latexsym}
\usepackage{amsmath}
\usepackage{amsthm}
\usepackage{graphicx}
\usepackage{txfonts}
\usepackage{bm}
\usepackage{color}
\usepackage{epic,eepic}

\newcommand{\RED}[1]{{\color{red}#1}} 
 \renewcommand{\RED}[1]{{#1}} 
\newcommand{\BLU}[1]{{\color{blue}#1}} 
 \renewcommand{\BLU}[1]{{#1}} 

\usepackage{geometry}
\numberwithin{equation}{section}

\newtheorem{theorem}{Theorem}[section]
\newtheorem{proposition}{Proposition}[section]
\newtheorem{corollary}{Corollary}[section]
\newtheorem{lemma}{Lemma}[section]
\newtheorem{remark}{Remark}[section]
\newtheorem{example}{Example}[section]

\newcommand{\memo}[1]{{\bf \small \RED{[MEMO:}} \BLU{#1} \ {\bf \small \RED{:end]} }}  
   \renewcommand{\memo}[1]{}           
\newcommand{\OMIT}[1]{{\bf [OMIT:} #1 \ {\bf --- end OMIT] }}  
   \renewcommand{\OMIT}[1]{}            

\newcommand{\llceil}{\bigg\lceil} 
\newcommand{\rrceil}{\bigg\rceil} 

\newcommand{\llfloor}{\bigg\lfloor} 
\newcommand{\rrfloor}{\bigg\rfloor} 

\newcommand{\RR}{{\mathbb{R}}}
\newcommand{\ZZ}{{\mathbb{Z}}}

\newcommand{\veczero}{{\bf 0}}
\newcommand{\dom}{{\rm dom\,}}

\newcommand{\argmax}{\arg \max}
\newcommand{\argmin}{\arg \min}

\newcommand{\subgR}{\partial}
\newcommand{\Lnat}{{L$^{\natural}$}}
\newcommand{\Mnat}{{M$^{\natural}$}}
\newcommand{\LLnat}{{L$^{\natural}_{2}$}}
\newcommand{\MMnat}{{M$^{\natural}_{2}$}}

\newcommand{\finbox}{\hspace*{\fill}$\rule{0.2cm}{0.2cm}$}
\newcommand{\todaye}{\the\year/\the\month/\the\day}

\begin{document}

\title{
Discrete Fenchel Duality for 
a Pair of Integrally Convex and Separable Convex Functions
}

\author{
Kazuo Murota%
\thanks{
The Institute of Statistical Mathematics,
and
Tokyo Metropolitan University, 
murota@tmu.ac.jp}
\ and 
Akihisa Tamura%
\thanks{Department of Mathematics, Keio University, 
aki-tamura@math.keio.ac.jp}
}

\date{August 2021 / October 2021 / December 2021}

\maketitle

\begin{abstract}
Discrete Fenchel duality is one of the central issues in discrete convex analysis.
The Fenchel-type min-max theorem for a pair of integer-valued \Mnat-convex functions 
generalizes the min-max formulas for polymatroid intersection
and valuated matroid intersection.
In this paper we establish a Fenchel-type min-max formula
for a pair of integer-valued integrally convex and separable convex functions.
Integrally convex functions 
constitute a fundamental function class
in discrete convex analysis,
including both \Mnat-convex functions and \Lnat-convex functions,
whereas separable convex functions are characterized as
those functions which are both \Mnat-convex and \Lnat-convex.
The theorem is proved by revealing a kind of box 
integrality of subgradients 
of an integer-valued integrally convex function.
The proof is based on the Fourier--Motzkin elimination.
\end{abstract}




\newpage


\section{Introduction}
\label{SCintro}

Discrete Fenchel duality is one of the central issues in discrete convex analysis
 \cite{Fuj05book,Mdca98,Mdcasiam,Mbonn09,Mdcaeco16}.
In this paper we establish a Fenchel-type min-max formula
for a pair of integer-valued integrally convex and separable convex functions.

Integrally convex functions,
due to Favati--Tardella \cite{FT90},
constitute a fundamental function class
in discrete convex analysis,
and almost all kinds of discrete convex functions 
are known to be integrally convex. 
Indeed, separable convex,
{\rm L}-convex, \Lnat-convex, {\rm M}-convex,  
\Mnat-convex,  \LLnat-convex, and 
\MMnat-convex functions are known to be integrally convex \cite{Mdcasiam}.
Multimodular functions \cite{Haj85} 
are also integrally convex, 
as pointed out in \cite{Mdcaprimer07}.
Moreover, BS-convex and UJ-convex functions \cite{Fuj14bisubmdc}
are integrally convex.
Discrete midpoint convex functions 
\cite{MMTT20dmc}
and directed discrete midpoint convex functions 
\cite{TT21ddmc} are also integrally convex.

The concept of integral convexity
is used in formulating discrete fixed point theorems
and found applications in economics and game theory
\cite{IMT05,Mdcaeco16,Yan09fixpt}.
A proximity theorem for integrally convex functions
is established in \cite{MMTT19proxIC} 
together with a proximity-scaling algorithm for minimization.
Fundamental operations for integrally convex functions
such as projection and convolution are investigated in \cite{MM19projcnvl,Msurvop21,Mopernet21}.
Integer-valued integrally convex functions
enjoy integral biconjugacy \cite{MT20subgrIC}.
Section \ref{SCintcnvfn} of this paper describes the definition and 
technical properties of integrally convex functions that we need in this paper.

The discrete Fenchel-type min-max theorem is formulated
in terms of conjugate functions.
For an integer-valued function 
$f: \ZZ\sp{n} \to \ZZ \cup \{ +\infty \}$
with $\dom f \ne \emptyset$,
we define
$f\sp{\bullet}: \ZZ\sp{n} \to \ZZ \cup \{ +\infty \}$ 
by
\begin{align}
f\sp{\bullet}(p)  &= \max\{  \langle p, x \rangle - f(x)   \mid x \in \ZZ\sp{n} \}
\qquad ( p \in \ZZ\sp{n}),
 \label{conjvexZpZ} 
\end{align}
called the {\em integral conjugate} of $f$.
Here,
$\langle p, x \rangle = \sum_{i=1}\sp{n} p_{i} x_{i}$
is the inner product of 
$p=(p_{1}, p_{2}, \ldots, p_{n})$ and 
$x=(x_{1}, x_{2}, \allowbreak \ldots, \allowbreak  x_{n})$,
and
for a function 
$h: \ZZ\sp{n} \to \RR \cup \{ -\infty, +\infty \}$ 
in general, 
$\dom h = \{ x \in \ZZ\sp{n} \mid -\infty < h(x) < +\infty \}$
is called the {\em effective domain} of $h$.
Note that $f\sp{\bullet}(p)$ may be $+\infty $ 
and hence using ``$\sup$'' (supremum) in  \eqref{conjvexZpZ} 
would be formally more precise but 
we choose to use ``$\max$'' (maximum).
The concave version of the conjugate function can also be defined 
for an integer-valued function 
$g: \ZZ\sp{n} \to \ZZ \cup \{ -\infty \}$.
Namely, we define
$g\sp{\circ}: \ZZ\sp{n} \to \ZZ \cup \{ -\infty \}$ 
by
\begin{align}
g\sp{\circ}(p)  &= \min\{  \langle p, x \rangle - g(x)   \mid x \in \ZZ\sp{n} \}
\qquad ( p \in \ZZ\sp{n}),
 \label{conjcavZpZ} 
\end{align}
and call this function $g\sp{\circ}$ the 
{\em integral concave conjugate} of $g$,
while $f\sp{\bullet}$ is called the 
{\em integral convex conjugate} of $f$
to distinguish between convex and concave versions.
We have $g\sp{\circ}(p) = -f\sp{\bullet}(-p)$ if $g(x) = -f(x)$.

The discrete Fenchel-type duality theorem 
asserts the min-max formula
\begin{align} 
 \min \{ f(x) - g(x) \mid  x \in \ZZ\sp{n}  \} 
= \max \{ g\sp{\circ}(p) - f\sp{\bullet}(p) \mid  p \in \ZZ\sp{n} \}
\label{minmaxGenZZ0} 
\end{align} 
for a pair of functions
$f: \ZZ\sp{n} \to \ZZ \cup \{ +\infty \}$
and 
$g: \ZZ\sp{n} \to \ZZ \cup \{ -\infty \}$.
It is supposed that $f$ and $g$ are equipped with certain 
discrete convexity and concavity 
and some additional regularity conditions are assumed.
Such a theorem
can be traced back to the Fenchel-type duality theorem 
for submodular set functions by Fujishige \cite{Fuj84fenc}
(see \cite[Theorem 6.3]{Fuj05book}).
The Fenchel-type min-max theorem for 
a pair of \Mnat-convex functions 
by Murota \cite{Mdca98} 
(see \cite[Theorem 8.21]{Mdcasiam})
generalizes the min-max formulas for polymatroid intersection,
valuated matroid intersection, and 
the Fenchel-type duality theorem for submodular set functions
 (see \cite[Section 8.23]{Mdcasiam}).
As is well known, the existence of such min-max formula
guarantees the existence of a certificate of optimality
for the problem of minimizing 
$f(x) - g(x)$ over $x \in \ZZ\sp{n}$.

The main result of this paper (Theorem~\ref{THminmaxICfnSpfnZZ} below) is
the Fenchel-type min-max formula \eqref{minmaxGenZZ0}
where $f$ is an integer-valued integrally convex function
and  $g$ is an integer-valued separable concave function.
A function
$\Psi: \ZZ^{n} \to \ZZ \cup \{ -\infty \}$
in $x=(x_{1}, x_{2}, \ldots,x_{n}) \in \ZZ^{n}$
is called  
{\em separable concave}
if it can be represented as
\begin{equation}  \label{sepcavdef}
\Psi(x) = \psi_{1}(x_{1}) + \psi_{2}(x_{2}) + \cdots + \psi_{n}(x_{n})
\end{equation}
with univariate discrete concave functions
$\psi_{i}: \ZZ \to \ZZ \cup \{ -\infty \}$,
which means, by definition, that 
$\dom \psi_{i}$ is an interval of integers and
\begin{equation}  \label{univarcavedef}
\psi_{i}(k-1) + \psi_{i}(k+1) \leq 2 \psi_{i}(k)
\qquad (k \in \ZZ).
\end{equation}
The integral concave conjugate of $\Psi$ 
is given by 
\begin{equation} \label{Psiconj00}
\Psi\sp{\circ}(p)= 
 \psi_{1}\sp{\circ}(p_{1}) + \psi_{2}\sp{\circ}(p_{2}) + \cdots + \psi_{n}\sp{\circ}(p_{n}) 
\end{equation}
where
\begin{equation} \label{psiconjdef00}
\psi_{i}\sp{\circ}(\ell)  = \min\{ k \ell  -  \psi_{i}(k)  \mid  k \in \ZZ \}
\qquad
(\ell \in \ZZ).
\end{equation}

\begin{theorem}[Main result] \label{THminmaxICfnSpfnZZ}
For an integer-valued integrally convex function
$f: \ZZ\sp{n} \to \ZZ \cup \{ +\infty \}$
with $\dom f \ne \emptyset$
and an integer-valued separable concave function
$\Psi: \ZZ\sp{n} \to \ZZ \cup \{ -\infty \}$
with $\dom \Psi \ne \emptyset$,
we have
\begin{align} 
 \min \{ f(x) - \Psi(x) \mid  x \in \ZZ\sp{n}  \} 
= \max \{ \Psi\sp{\circ}(p) - f\sp{\bullet}(p) \mid  p \in \ZZ\sp{n} \},
\label{minmaxICfnSpfnZZ00} 
\end{align} 
where the minimum or the maximum is assumed to be finite. 
\finbox
\end{theorem}

For example, suppose that $f(x)$ represents a certain loss function 
in an integer vector $x$
and we want to minimize the loss $f(x)$ with an additional term for regularization
such as the $\ell_{1}$-norm $\| x \|_{1}$ and the squared $\ell_{2}$-norm 
$\| x \|_{2}\sp{2}$.
Then our problem is to minimize
$ f(x) + C \| x \|_{1}$ or
$ f(x) + C \| x \|_{2}\sp{2}$
with $C > 0$,
which is in the form of minimizing 
$f(x) - \Psi(x)$ over $x \in \ZZ\sp{n}$
with a separable concave function $\Psi(x)$.
If the loss function $f(x)$ can be chosen to be an integer-valued integrally convex function,
which does not seem to be so restrictive, 
we may apply Theorem~\ref{THminmaxICfnSpfnZZ}.
Note that Theorem~\ref{THminmaxICfnSpfnZZ}
implies the existence of a certificate of optimality
for the problem of minimizing 
$f(x) - \Psi(x)$ over $x \in \ZZ\sp{n}$; see Section \ref{SCfencthmImplic} for detail.

\medskip

We prove Theorem~\ref{THminmaxICfnSpfnZZ}
by revealing a kind of box integrality of subgradients 
of an integer-valued integrally convex function.
Let $f: \ZZ\sp{n} \to \ZZ \cup \{ +\infty \}$
be an integer-valued function.
The {\em subdifferential} of $f$ at $x \in \dom f$
is defined as 
\begin{equation} \label{subgZRdef00}
 \subgR f(x)
= \{ p \in  \RR\sp{n} \mid    
  f(y) - f(x)  \geq  \langle p, y - x \rangle   \ \ \mbox{\rm for all }  y \in \ZZ\sp{n} \}
\end{equation}
and an element $p$ of $\subgR f(x)$ is called a 
{\em subgradient}
of $f$ at $x$.
An integer vector $p$ belonging to $\subgR f(x)$ 
is called an {\em integral subgradient},
and the condition 
\begin{equation} \label{ICsubgrZ00}
\subgR f(x) \cap \ZZ\sp{n} \neq \emptyset
\end{equation}
is sometimes referred to as
the {\em integral subdifferentiability}  of $f$ at $x$.
It has been shown  
by Murota--Tamura \cite{MT20subgrIC} 
that \eqref{ICsubgrZ00} holds for 
any integer-valued integrally convex function $f$
and $x \in \dom f$.

Our proof of Theorem~\ref{THminmaxICfnSpfnZZ}
is based on a strengthening of 
integral subdifferentiability \eqref{ICsubgrZ00}
with an additional box condition.
For two integer vectors 
$\alpha \in (\ZZ \cup \{ -\infty \})\sp{n}$ and 
$\beta \in (\ZZ \cup \{ +\infty \})\sp{n}$ with $\alpha \leq \beta$,
we define notation
\[
 [\alpha,\beta]_{\RR} = \{ p \in \RR\sp{n} \mid \alpha \leq p \leq \beta \},
\]
which represents the set of real vectors between $\alpha$ and $\beta$. 
An {\em integral box} will mean
a set $B$ of real vectors represented as
\[
  B =  [\alpha,\beta]_{\RR} = \{ p \in \RR\sp{n} \mid \alpha \leq p \leq \beta \}
\]
for some $\alpha \in (\ZZ \cup \{ -\infty \})\sp{n}$ and 
$\beta \in (\ZZ \cup \{ +\infty \})\sp{n}$ with $\alpha \leq \beta$.

Our main technical result is the following,
with which our main result (Theorem~\ref{THminmaxICfnSpfnZZ}) is proved 
in Section \ref{SCintdualopt}.

\begin{theorem}[Main technical result] \label{THICsubgrBox}
Let $f: \ZZ\sp{n} \to \ZZ \cup \{ +\infty \}$
be an integer-valued integrally convex function, $x \in \dom f$,
and $B$ be an integral box.
If $\subgR f(x) \cap B$ is nonempty,
then $\subgR f(x) \cap B$ 
is a polyhedron containing an integer vector.
If, in addition, $\subgR f(x) \cap B$ is bounded,
then $\subgR f(x) \cap B$ has an integral vertex.
\finbox
\end{theorem}
\noindent
The content of the above theorem may be expressed succinctly as:
\begin{equation} \label{ICsubgrboxZ00}
\subgR f(x) \cap B \neq \emptyset
\ \Longrightarrow \ 
\subgR f(x) \cap B \cap \ZZ\sp{n} \neq \emptyset.
\end{equation}

This paper is organized as follows.
Section~\ref{SCintcnvfn} is a review of relevant results on
integrally convex functions.
Section~\ref{SCfencthm} 
presents the Fenchel-type min-max formula
for a pair of integer-valued integrally convex and separable convex functions
(Theorem~\ref{THminmaxICfnSpfnZZ}) 
with its implications and significances in discrete convex analysis
as well as the derivation of the theorem 
from Theorem~\ref{THICsubgrBox}.
Section~\ref{SCsubrZ} establishes 
the main technical result 
(Theorem~\ref{THICsubgrBox})
by means of the Fourier--Motzkin elimination.

\section{Integrally Convex Functions}
\label{SCintcnvfn}

In this section we summarize fundamental facts about integrally convex functions.

For $x \in \RR^{n}$ the integral neighborhood of $x$ is defined as 
\begin{equation}  \label{intneighbordef}
N(x) = \{ z \in \ZZ^{n} \mid | x_{i} - z_{i} | < 1 \ (i=1,2,\ldots,n)  \}.
\end{equation}
It is noted that 
strict inequality ``\,$<$\,'' is used in this definition
and hence $N(x)$ admits an alternative expression
\begin{equation}  \label{intneighbordeffloorceil}
N(x) = \{ z \in \ZZ\sp{n} \mid
\lfloor x_{i} \rfloor \leq  z_{i} \leq \lceil x_{i} \rceil  \ \ (i=1,2,\ldots, n) \} ,
\end{equation}
where, for $t \in \RR$ in general, 
$\left\lfloor  t  \right\rfloor$
denotes the largest integer not larger than $t$
(rounding-down to the nearest integer)
and 
$\left\lceil  t   \right\rceil$ 
is the smallest integer not smaller than $t$
(rounding-up to the nearest integer).
For a set $S \subseteq \ZZ^{n}$
and $x \in \RR^{n}$
we call the convex hull of $S \cap N(x)$ 
the {\em local convex hull} of $S$ at $x$.
A nonempty set
$S \subseteq \ZZ^{n}$ is said to be 
{\em integrally convex} if
the union of the local convex hulls $\overline{S \cap N(x)}$ over $x \in \RR^{n}$ 
is convex \cite{Mdcasiam}.
This is equivalent to saying that,
for any $x \in \RR^{n}$, 
$x \in \overline{S} $ implies $x \in  \overline{S \cap N(x)}$.
An integrally convex set $S$ is ``hole-free'' in the sense that 
$S =  \overline{S} \cap \mathbb{Z}^{n}$.

It is recognized only recently
\cite{Mopernet21}
 that the concept of integrally convex sets is 
closely related (or essentially equivalent) to 
the concept of box-integer polyhedra.
Recall from \cite[Section~5.15]{Sch03} 
that a polyhedron $P \subseteq \RR\sp{n}$ is called 
{\em box-integer}
if $P \cap \{ x \in \RR\sp{n} \mid a  \leq x \leq b \}$
is an integer polyhedron for each choice of integer vectors $a$ and $b$.
It is easy to see that 
if a set
$S \subseteq \ZZ^{n}$
is integrally convex, then its convex hull
$\overline{S}$ is a box-integer polyhedron, and conversely,
if $P$ is a box-integer polyhedron, then 
$S = P \cap \ZZ\sp{n}$ is an integrally convex set.

For a function $f: \ZZ^{n} \to \RR \cup \{ +\infty  \}$
with $\dom f \ne \emptyset$,
the {\em convex envelope} of $f$ means
the (point-wise) largest convex function 
$g: \RR^{n} \to \RR \cup \{ +\infty  \}$
that satisfies $g(x) \leq f(x)$ for all $x \in \ZZ^{n}$.
The convex envelope of $f$ is denoted by $\overline{f}$.
If $\overline{f}(x) = f(x)$ for all $x \in \ZZ^{n}$,
we call $f$ {\em convex extensible}
and refer to $\overline{f}$ also as the {\em convex extension} of $f$.

Let $f: \ZZ^{n} \to \RR \cup \{ +\infty  \}$
be a function with $\dom f \ne \emptyset$.
The {\em local convex extension} 
$\tilde{f}: \RR^{n} \to \RR \cup \{ +\infty \}$
of $f$ is defined 
as the union of all convex envelopes of $f$ on $N(x)$.  That is,
\begin{equation} \label{fnconvclosureloc2}
 \tilde f(x) = 
  \min\{ \sum_{y \in N(x)} \lambda_{y} f(y) \mid
      \sum_{y \in N(x)} \lambda_{y} y = x,  \ 
  (\lambda_{y})  \in \Lambda(x) \}
\quad (x \in \RR^{n}) ,
\end{equation} 
where $\Lambda(x)$ denotes the set of coefficients for convex combinations indexed by $N(x)$:
\begin{equation} \label{Lambdadef}
  \Lambda(x) = \{ (\lambda_{y} \mid y \in N(x) ) \mid 
      \sum_{y \in N(x)} \lambda_{y} = 1, 
      \lambda_{y} \geq 0 \ \ \mbox{for all } \   y \in N(x)  \} .
\end{equation} 
If $\tilde f$ is convex on $\RR^{n}$,
then $f$ is said to be {\em integrally convex}
\cite{FT90}.
In this case we have
$\tilde f (x) = \overline{f}(x)$ for all $x \in \RR^{n}$.
The effective domain of an integrally convex function is an integrally convex set.
A set $S \subseteq \ZZ\sp{n}$ is integrally convex if and only if its indicator function
$\delta_{S}: \ZZ\sp{n} \to \{ 0, +\infty \}$
defined by
\[
\delta_{S}(x)  =
   \left\{  \begin{array}{ll}
    0            &   (x \in S) ,      \\
   + \infty      &   (x \not\in S)  \\
                      \end{array}  \right.
\]
is integrally convex.

Integral convexity of a function can be characterized as follows.
The condition (c) below is 
a local condition under the assumption that the effective domain is an integrally convex set.

\begin{theorem}[\cite{FT90,MMTT19proxIC,MMTT20dmc}]
\label{THfavtarProp33}
Let $f: \mathbb{Z}^{n} \to \mathbb{R} \cup \{ +\infty  \}$
be a function 
with $\dom f \neq \emptyset$.
Then the following properties are equivalent,
where $\tilde{f}$ is the local convex extension of $f$ 
defined by \eqref{fnconvclosureloc2}.

{\rm (a)}
$f$ is integrally convex.

{\rm (b)}
For every $x, y \in \ZZ\sp{n}$ with $\| x - y \|_{\infty} \geq 2$  we have \ 
\begin{equation}  \label{intcnvconddist2}
\tilde{f}\, \bigg(\frac{x + y}{2} \bigg) 
\leq \frac{1}{2} (f(x) + f(y)).
\end{equation}

{\rm (c)}
The effective domain $\dom f$ is an integrally convex set, and 
\eqref{intcnvconddist2} holds
for every $x, y \in \ZZ\sp{n}$ with $\| x - y \|_{\infty} =2$.
\finbox
\end{theorem}

\begin{remark} \rm \label{RMintcnvconcept}
The concept of integrally convex functions is introduced in \cite{FT90} 
for functions defined on integer intervals (discrete rectangles).
The extension to functions with general integrally convex effective domains
is straightforward, which is found in \cite{Mdcasiam}.
Theorem~\ref{THfavtarProp33} originates in \cite[Proposition 3.3]{FT90},
which shows the  equivalence of (a) and (c)
when the effective domain is an integer interval (box),
while the equivalence of (a) and (c) 
for a general integral convex effective domain 
is proved in \cite[Appendix A]{MMTT19proxIC}. 
The equivalence of (a) and (b)
in Theorem~\ref{THfavtarProp33} is shown in 
\cite[Theorem A.1]{MMTT20dmc}.
\finbox
\end{remark}

\begin{example} \rm  \label{EXsepIC}
A function
$\Phi: \ZZ^{n} \to \RR \cup \{ +\infty \}$
in $x=(x_{1}, x_{2}, \ldots,x_{n}) \in \ZZ^{n}$
is called  
{\em separable convex}
if it can be represented as
\begin{equation}  \label{sepvexdef}
\Phi(x) = \varphi_{1}(x_{1}) + \varphi_{2}(x_{2}) + \cdots + \varphi_{n}(x_{n})
\end{equation}
with univariate discrete convex functions
$\varphi_{i}: \ZZ \to \RR \cup \{ +\infty \}$, 
which means, by definition, that 
$\dom \varphi_{i}$ is an interval of integers and
\begin{equation}  \label{univarvexdef}
\varphi_{i}(k-1) + \varphi_{i}(k+1) \geq 2 \varphi_{i}(k)
\qquad (k \in \ZZ).
\end{equation}
A separable convex function is integrally convex.
\finbox
\end{example}

\begin{example} \rm  \label{EXdiagdomIC}
A symmetric matrix $Q =(q_{ij})$ that satisfies the condition
\begin{equation}\label{midptdiagdomdef}
q_{ii} \geq \sum_{j \neq i} |q_{ij}|
\qquad (i=1,\ldots,n)
\end{equation}
is called a diagonally dominant matrix (with nonnegative diagonals).
If  $Q$ is diagonally dominant
in the sense of \eqref{midptdiagdomdef},
then $f(x) = x\sp{\top} Q x$ is integrally convex
\cite[Proposition 4.5]{FT90}.
The converse is also true if $n \leq 2$
\cite[Remark 4.3]{FT90}.
\finbox
\end{example}

\begin{example} \rm \label{EXtwosepIC}
A function 
$f: \ZZ^{n} \to \RR \cup \{ +\infty \}$
is called
{\em 2-separable convex}
if it can be expressed as the sum of univariate convex,
diff-convex, and sum-convex functions, i.e., if
\begin{equation}\label{twosepSDconv}
f(x_1, \ldots, x_n) = 
\sum_{i=1}\sp{n} \varphi_{i}(x_{i}) 
+ \sum_{i \neq j}  \varphi_{ij}(x_{i} - x_{j}) + \sum_{i \neq j} \psi_{ij}(x_{i}+x_{j})  ,
\end{equation}
where
$\varphi_{i}, \varphi_{ij}, \psi_{ij}: \mathbb{Z} \to \mathbb{R} \cup \{ +\infty \}$
$(i, j =1,\ldots,n; \  i \not = j)$
are univariate convex functions.
A 2-separable convex function is known \cite{TT21ddmc} to be integrally convex.
A quadratic function $f(x) = x\sp{\top} Q x$ 
with $Q$ satisfying \eqref{midptdiagdomdef}
is an example of a 2-separable convex function.
\finbox
\end{example}

A minimizer of an integrally convex function
can be characterized by a local minimality condition as follows.

\begin{theorem}[\protect{\cite[Proposition 3.1]{FT90}};
  see also \protect{\cite[Theorem 3.21]{Mdcasiam}}]
  \label{THintcnvlocopt}
Let $f: \mathbb{Z}^{n} \to \mathbb{R} \cup \{ +\infty  \}$
be an integrally convex function and $x^{*} \in \dom f$.
Then $x^{*}$ is a minimizer of $f$ 
if and only if
$f(x^{*}) \leq f(x^{*} +  d)$ for all $d \in  \{ -1, 0, +1 \}^{n}$.
\finbox
\end{theorem}

We need the following fact for the proof of Theorem~\ref{THminmaxICfnSpfnZZ}.
Recall that the convex envelope of a function $f$ is denoted by $\overline{f}$,
which coincides with $\tilde f$ in \eqref{fnconvclosureloc2}
if $f$ is integrally convex.
It is noted that the statement cannot be extended to a pair of 
general integrally convex functions.

\begin{proposition} \label{PRcnvextICsep}
{\rm (1)}
Let $f: \mathbb{Z}^{n} \to \mathbb{R} \cup \{ +\infty  \}$
be an integrally convex function,
and $\Phi: \mathbb{Z}^{n} \to \mathbb{R} \cup \{ +\infty  \}$
a separable convex function.
Then
$\overline{f + \Phi} = \overline{f} + \overline{\Phi}$.

{\rm (2)}
Let $S \subseteq \ZZ\sp{n}$ 
be an integrally convex set and 
$D = \{ x \in \ZZ\sp{n} \mid \alpha \leq x \leq  \beta \}$
for some
$\alpha \in (\ZZ \cup \{ -\infty \})\sp{n}$ and
$\beta \in (\ZZ \cup \{ +\infty \})\sp{n}$ with $\alpha \leq \beta$.
Then
$\overline{S \cap D} = \overline{S} \cap \overline{D}$.
\end{proposition}

\begin{proof}
(1)
Fix $x \in \RR\sp{n}$.
Since $f$ is integrally convex, we have 
$\overline{f}(x) = \sum_{y \in N(x)} \lambda_{y} f(y)$
for some $\lambda \in \Lambda(x)$ (cf., \eqref{Lambdadef} for notation).
By separable convexity, we have
$\overline{\Phi}(x) = \sum_{y \in N(x)} \lambda_{y} \Phi(y)$
with the same coefficient $\lambda$.
This implies
\[
\overline{f}(x) +\overline{\Phi}(x) 
= \sum_{y \in N(x)} \lambda_{y} (f(y)+\Phi(y))
\geq (\overline{f + \Phi})(x),
\]
while the reverse inequality
$\overline{f}(x) +\overline{\Phi}(x) \leq (\overline{f + \Phi})(x)$
is obvious from the definition of convex envelopes.

(2)
This follows from (1) with $f = \delta_{S}$
and
$\Phi = \delta_{D}$.
Note that 
$f + \Phi = \delta_{S \cap D}$,
$\overline{f} = \delta_{\overline{S}}$,
$\overline{\Phi} = \delta_{\overline{D}}$, etc.
\end{proof}

The integral conjugate $f\sp{\bullet}$ of a function
$f: \ZZ\sp{n} \to \ZZ \cup \{ +\infty \}$
is also an integer-valued function defined on $\ZZ\sp{n}$.
So we can apply the transformation (\ref{conjvexZpZ}) 
to $f\sp{\bullet}$  
to obtain
$f\sp{\bullet\bullet} = (f\sp{\bullet})\sp{\bullet}$,
which is called the {\em integral biconjugate} of $f$.
Although the integral conjugate $f\sp{\bullet}$ 
of an integer-valued integrally convex function $f$
is not necessarily integrally convex
(\cite[Example 4.15]{MS01rel}, \cite[Remark 2.3]{MT20subgrIC}),
it is known \cite{MT20subgrIC} that
the integral biconjugate $f\sp{\bullet\bullet}$
coincides with $f$ itself.

\begin{theorem}[\cite{MT20subgrIC}]  \label{THbiconjIC}
For every integer-valued integrally convex function 
$f: \ZZ^{n} \to \ZZ \cup \{ +\infty \}$
with $\dom f \ne \emptyset$,
we have
$f\sp{\bullet\bullet}(x) =f(x)$ for all $x \in \ZZ\sp{n}$.
\finbox
\end{theorem}

The reader is referred to 
\cite{MM19projcnvl,MMTT19proxIC,MT20subgrIC}
for recent developments
in the theory of integral convexity,
and to \cite[Section 3.4]{Mdcasiam} for basic facts about integral convexity.

\section{Discrete Fenchel Duality Theorem}
\label{SCfencthm}

\subsection{Main theorem and its implications}
\label{SCfencthmImplic}

In this section we address the main result of this paper,
which has already been presented in Introduction 
as Theorem~\ref{THminmaxICfnSpfnZZ},
where the proof will be given in Section \ref{SCproofminmax}.
Recall the notations
\begin{align}
f\sp{\bullet}(p) 
& = \max \{ \langle p, x \rangle  - f(x) \mid  x \in \ZZ\sp{n} \}
\qquad ( p\in \ZZ\sp{n}),
 \label{conjvexZpZ2} 
\\
\Psi\sp{\circ}(p) & = 
\min \{ \langle p, x \rangle - \Psi(x) \mid x \in \ZZ\sp{n} \}
\qquad ( p\in \ZZ\sp{n}),
 \label{conjcavZpZ2} 
\end{align}
as well as the expressions 
of $\Psi$ in \eqref{sepcavdef} 
and $\Psi\sp{\circ}$ in \eqref{Psiconj00}.

\medskip
\noindent
{\bf Theorem~\ref{THminmaxICfnSpfnZZ}.} (Main result, again) \ 
\textit{
For an integer-valued integrally convex function
$f: \ZZ\sp{n} \to \ZZ \cup \{ +\infty \}$
with $\dom f \ne \emptyset$
and an integer-valued separable concave function
$\Psi: \ZZ\sp{n} \to \ZZ \cup \{ -\infty \}$
with $\dom \Psi \ne \emptyset$,
we have
\begin{align} 
 \min \{ f(x) - \Psi(x) \mid  x \in \ZZ\sp{n}  \} 
= \max \{ \Psi\sp{\circ}(p) - f\sp{\bullet}(p) \mid  p \in \ZZ\sp{n} \} ,
\label{minmaxICfnSpfnZZ} 
\end{align} 
where the minimum or the maximum is assumed to be finite. 
}
\finbox

\begin{remark} \rm \label{RMminmaxAssmp}
The assumption on the left-hand side of \eqref{minmaxICfnSpfnZZ} being finite
means that 
$\dom f \cap \dom \Psi \neq \emptyset$ and 
the set $\{ f(x) - \Psi(x) \mid  x \in \ZZ\sp{n}  \}$
of function values is bounded from below.
Since the function $f(x) - \Psi(x)$ is integer-valued,
this assumption immediately implies that 
there exists $x$ that attains the minimum.
It will be shown in Lemma~\ref{LMfntmaxfntmin} in Section~\ref{SCfntmaxfntmin}
that, if the maximum on the right-hand side of \eqref{minmaxICfnSpfnZZ} 
is finite, then 
the minimum on the left-hand side is also finite.
\finbox
\end{remark}

\begin{remark} \rm \label{RMlinfn}
When $\Psi$ is a linear function, say, 
$\Psi(x) = \langle c, x \rangle$ with $c \in \ZZ\sp{n}$,
the formula \eqref{minmaxICfnSpfnZZ} reduces to a triviality.
Indeed, in this case we have 
$\Psi\sp{\circ}(p) = 0$ for $p = c$ and
$\Psi\sp{\circ}(p) =-\infty$ for $p \neq c$,
and hence
\begin{align*}
& \mbox{LHS of \eqref{minmaxICfnSpfnZZ}}
= \min \{ f(x) - \langle c, x \rangle  \mid  x \in \ZZ\sp{n} \}
= - f\sp{\bullet}(c),
\\
& \mbox{RHS of \eqref{minmaxICfnSpfnZZ}} =
\max \{ 0 - f\sp{\bullet}(p) \mid  p=c \} = - f\sp{\bullet}(c).
\end{align*}
The formula \eqref{minmaxICfnSpfnZZ} is also a triviality
when $\dom f \subseteq \{ 0,1 \}\sp{n}$.
In this case, we may assume 
$\Psi(x) = \langle c, x \rangle$ with $\dom \Psi = \ZZ\sp{n}$,
and the above argument applies.
In this connection it is recalled that every function 
$f: \ZZ^{n} \to \ZZ \cup \{ +\infty \}$
with $\dom f \subseteq \{ 0,1 \}\sp{n}$ is integrally convex.
\finbox
\end{remark}

Theorem~\ref{THminmaxICfnSpfnZZ} implies
a min-max theorem for separable convex minimization
on a box-integer polyhedron 
(see Section \ref{SCintcnvfn} for the definition of a box-integer polyhedron).

\begin{theorem}  \label{THboxint}
For a nonempty box-integer polyhedron
$P \ (\subseteq \RR\sp{n})$
and an integer-valued separable convex function
$\Phi: \ZZ\sp{n} \to \ZZ \cup \{ +\infty \}$
with $\dom \Phi \ne \emptyset$,
we have
\begin{align} 
 \min \{ \Phi (x) \mid  x \in P \cap \ZZ\sp{n} \} 
= \max \{ \mu(p) - \Phi\sp{\bullet}(p) \mid  p\in \ZZ\sp{n}\} ,
\label{minmaxZZintcnvA} 
\end{align} 
where
$\mu(p)= \min \{ \langle p, x \rangle \mid  x\in P \}$
and 
the minimum or the maximum in \eqref{minmaxZZintcnvA} is assumed to be finite. 
\end{theorem}
\begin{proof}
Denote the indicator function of $P \cap \ZZ\sp{n}$ by $\delta$,
which is an integer-valued integrally convex function
because $P$ is a box-integer polyhedron.
With the use of $\Psi(x) = -\Phi(x)$ we have
\[
\mbox{LHS of \eqref{minmaxZZintcnvA}}
=  \min \{ \Phi (x) \mid  x\in P \cap \ZZ\sp{n} \} 
=  \min \{ \delta(x)  -\Psi (x) \mid  x\in \ZZ\sp{n} \} .
\]
On the other hand, on noting
\begin{align*}
& \mu(p) 
= \min \{ \langle p, x \rangle \mid  x\in P\}
= - \max \{ \langle -p, x \rangle \mid x \in P \cap \ZZ\sp{n} \}
= - \delta\sp{\bullet}(-p),
\\ &
\Phi\sp{\bullet}(p)
= \max \{ \langle p, x \rangle - \Phi(x) \}
= - \min \{ \langle -p, x \rangle - \Psi(x) \}
= - \Psi\sp{\circ}(-p),
\end{align*}
we obtain
\begin{align*}
 \mbox{RHS of \eqref{minmaxZZintcnvA}}
& = \max \{ \mu(p) - \Phi\sp{\bullet}(p) \mid  p\in \ZZ\sp{n}\}
\\ & 
= \max \{  \Psi\sp{\circ}(-p) - \delta\sp{\bullet}(-p) \mid  p\in \ZZ\sp{n}\}
\\ & 
= \max \{  \Psi\sp{\circ}(p) - \delta\sp{\bullet}(p) \mid  p\in \ZZ\sp{n}\}.
\end{align*}
Therefore,
\eqref{minmaxZZintcnvA} is equivalent to
\[
 \min \{ \delta(x)  -\Psi (x) \mid  x\in  \ZZ\sp{n} \} 
= \max \{  \Psi\sp{\circ}(p) - \delta\sp{\bullet}(p) \mid  p\in \ZZ\sp{n}\},
\]
which is a special case of \eqref{minmaxICfnSpfnZZ} in Theorem~\ref{THminmaxICfnSpfnZZ}.
\end{proof}

\begin{remark} \rm \label{RMboxTDI}
Theorem~\ref{THboxint}
generalizes a recent result of Frank--Murota \cite[Theorem 3.4]{FM20boxTDI},
which asserts the min-max formula \eqref{minmaxZZintcnvA}
when $P$ is a box-TDI polyhedron
and the minimum is finite.
Note that  
a box-TDI polyhedron is a special case of a box-integer polyhedron.
See, e.g., \cite{Sch03} for the definition of a box-TDI polyhedron.
\finbox
\end{remark}

Theorem~\ref{THminmaxICfnSpfnZZ} also implies
a min-max theorem of Cunningham--Green-Kr{\'o}tki \cite{CG91degseq}
obtained in a study of $b$-matching degree-sequence polyhedra 
and the box convolution theorem of bisubmodular functions 
by Fujishige--Patkar \cite{FP94}.
We discuss this connection in Section \ref{SCbisubbox}.

Another Fenchel-type min-max formula can be obtained by combining
our main result
(Theorem~\ref{THminmaxICfnSpfnZZ})
with the biconjugacy theorem (Theorem~\ref{THbiconjIC}).
Let $\mathcal{G}$ denote the set of 
integral conjugates of integer-valued integrally convex functions.
By the biconjugacy theorem, we can alternatively say that 
$\mathcal{G}$ is the set of integer-valued functions $g$ whose 
integral conjugate is an integer-valued integrally convex function.
That is, 
\begin{align} 
 \mathcal{G} 
&= \{ g \mid g = f\sp{\bullet} \mbox{ for some integer-valued integrally convex $f$} \}
\nonumber \\
&= \{ g \mid  \mbox{$g\sp{\bullet}$ is an integer-valued integrally convex function} \}.
\label{conjclassIC}
\end{align}

\begin{theorem} \label{THminmaxICfnConjSpfnZZ}
For a function
$g: \ZZ\sp{n} \to \ZZ \cup \{ +\infty \}$
in $\mathcal{G}$ with $\dom g \ne \emptyset$
and an integer-valued separable concave function
$\Psi: \ZZ\sp{n} \to \ZZ \cup \{ -\infty \}$
with $\dom \Psi \ne \emptyset$,
we have
\begin{align} 
 \min \{ g(x) - \Psi(x) \mid  x \in \ZZ\sp{n}  \} 
= \max \{ \Psi\sp{\circ}(p) - g\sp{\bullet}(p) \mid  p \in \ZZ\sp{n} \},
\label{minmaxICfnConjSpfnZZ} 
\end{align} 
where the minimum or the maximum is assumed to be finite. 
\end{theorem}
\begin{proof}
First note that the integral concave conjugate 
$\Psi\sp{\circ}$ of $\Psi$ is also an integer-valued separable concave function.
By replacing $(f, \Psi)$ in \eqref{minmaxICfnSpfnZZ} with 
$(g\sp{\bullet}, \Psi\sp{\circ})$ we obtain
\[
 \min \{ g\sp{\bullet}(x) - \Psi\sp{\circ}(x) \mid  x \in \ZZ\sp{n}  \} 
= \max \{ \Psi\sp{\circ\circ}(p) - g\sp{\bullet\bullet}(p)\mid  p \in \ZZ\sp{n} \}.
\]
With the biconjugacy
$\Psi\sp{\circ\circ}= \Psi$ and $g\sp{\bullet\bullet}=g$,
where the former is well known and 
the latter is due to Theorem~\ref{THbiconjIC},
we can rewrite this formula to 
\[
 \min \{ g\sp{\bullet}(x) - \Psi\sp{\circ}(x) \mid  x \in \ZZ\sp{n}  \} 
= \max \{ \Psi(p) - g(p)\mid  p \in \ZZ\sp{n} \},
\]
which is equivalent to \eqref{minmaxICfnConjSpfnZZ}.
\end{proof}

\begin{remark} \rm \label{RMunivar}
It is often possible \cite{FM19partII,FM20boxTDI}
to obtain an explicit form of the integral conjugate function 
of an integer-valued separable convex (or concave) function.
For example, we have:
\begin{align}
& 
\mbox{If $\Phi(x) = C \| x \|_{1}$, then} \quad
\Phi\sp{\bullet}(p) =
\begin{cases} 
 0  &  ( \| p \|_{\infty} \leq C), 
\\ 
+\infty & (\mbox{\rm otherwise}),
\end{cases}
\\ &
\mbox{If $\Phi(x) = C \| x \|_{2}\sp{2}$, then} \quad
\Phi\sp{\bullet}(p) =
\sum_{i=1}\sp{n}
   \llfloor  \frac{p_{i} + C}{2 C}\rrfloor  \ 
 \left( p_{i} -  C \llfloor  \frac{p_{i} + C}{2 C}\rrfloor  \right). 
\end{align}
These expressions can be derived easily from the following facts.

{\rm (1)}
Let $\alpha$ be a nonnegative integer and $k_{0}$ an integer,
and define
\begin{equation*}  
 \varphi_{(1)} (k; \alpha, k_{0}) := \alpha |k - k_0 |  \quad (k\in {\bf Z}).  
\end{equation*}
The discrete conjugate of this function 
is given, for integers $\ell$, by 
\begin{equation*} 
\varphi_{(1)}\sp{\bullet} (\ell; \alpha, k_{0})
=
\begin{cases} 
 k_0 \ell  & (|\ell| \leq \alpha) , 
\\ 
+\infty &  (\mbox{\rm otherwise}). 
\end{cases}
\end{equation*}

{\rm (2)}
Let $\beta$ be a positive integer and 
$k_{0}$ an integer, and define
\begin{equation*}  
 \varphi_{(2)} (k; \beta, k_{0}) := \beta (k - k_{0})\sp{2} \quad (k\in {\bf Z}).  
\end{equation*}
The discrete conjugate of this function 
is given, for integers $\ell$, by 
\begin{equation*}
 \varphi_{(2)} \sp{\bullet} (\ell;\beta, k_{0})
 = k_{0} \ell +  \llfloor  \frac{\ell + \beta}{2 \beta}\rrfloor  \ 
 \left( \ell -  \beta \llfloor  \frac{\ell + \beta}{2\beta}\rrfloor  \right). 
\end{equation*}
When $\beta=1$ we have a simpler expression:
\begin{equation*}
 \varphi_{(2)} \sp{\bullet} (\ell; 1, k_{0}) = 
 k_{0} \ell + 
   \llfloor  \frac{\ell}{2}\rrfloor  
  \  \llceil \frac{\ell}{2}\rrceil .
\end{equation*}
\finbox
\end{remark}

\subsection{Fenchel duality for other function classes}
\label{SCfencthmComp}

The discrete Fenchel-type duality theorem, in its general form,
asserts the min-max formula
\begin{align} 
 \min \{ f(x) - g(x) \mid  x \in \ZZ\sp{n}  \} 
= \max \{ g\sp{\circ}(p) - f\sp{\bullet}(p) \mid  p \in \ZZ\sp{n} \}
\label{minmaxGenZZ} 
\end{align} 
under the assumption that $f$ and $g$ are equipped with 
certain specified discrete convexity and concavity. 
In this section we summarize our present knowledge
by compiling the results of this paper and
the known facts in discrete convex analysis \cite{Mdcasiam}.

To this end we introduce notations 
for classes of functions 
$f: \ZZ\sp{n} \to \ZZ \cup \{ +\infty \}$:
\begin{align*} 
 \mathcal{F} 
&= \{ f \mid \mbox{$f$ is an integer-valued integrally convex function} \},
\nonumber \\
 \mathcal{G} 
&= \{ f \mid \mbox{$f$ is the integral conjugate of an integer-valued integrally convex function} \},
\nonumber \\
 \mathcal{L} 
&= \{ f \mid \mbox{$f$ is an integer-valued \Lnat-convex function} \},
\nonumber \\
 \mathcal{M} 
&= \{ f \mid \mbox{$f$ is an integer-valued \Mnat-convex function} \},
\nonumber \\
 \mathcal{S}
&= \{ f \mid \mbox{$f$ is an integer-valued separable convex function} \}.
\end{align*}
We also use notation
$\mathcal{F}\sp{\bullet}$
for the set of integral conjugates of functions in $\mathcal{F}$;
similarly for
$\mathcal{G}\sp{\bullet}$, $\mathcal{L}\sp{\bullet}$, etc.
Then we have
\[
 \mathcal{F}\sp{\bullet} =  \mathcal{G},
\quad
 \mathcal{G}\sp{\bullet} =  \mathcal{F},
\quad
 \mathcal{L}\sp{\bullet} =  \mathcal{M},
\quad
 \mathcal{M}\sp{\bullet} =  \mathcal{L},
\quad
 \mathcal{S}\sp{\bullet} =  \mathcal{S},
\]
where the relations
$\mathcal{L}\sp{\bullet} =  \mathcal{M}$
and $\mathcal{M}\sp{\bullet} =  \mathcal{L}$
are known as the discrete conjugacy theorem \cite[Theorem 8.12]{Mdcasiam}.
We have the following inclusion relations:
\begin{equation} 
 \mathcal{F} \cap \mathcal{G} \supseteq \mathcal{L},
\qquad
 \mathcal{F} \cap \mathcal{G} \supseteq \mathcal{M},
\qquad
 \mathcal{L} \cap \mathcal{M}  =  \mathcal{S},
\label{FnclassIncl}
\end{equation} 
where 
$\mathcal{L} \cap \mathcal{M}  =  \mathcal{S}$
is stated in \cite[Theorem 8.49]{Mdcasiam}.

By combining Theorems \ref{THminmaxICfnSpfnZZ} and \ref{THminmaxICfnConjSpfnZZ} 
and the known facts \cite[Theorem 8.21]{Mdcasiam}
we obtain the following table to summarize our present knowledge
about the min-max formula \eqref{minmaxGenZZ}.
For example,
``Th.\ref{THminmaxICfnSpfnZZ}''
in the upper-right corner of the table
indicates that 
\eqref{minmaxGenZZ}
for $(f,g)$ with $f \in \mathcal{F}$ and $-g \in \mathcal{S}$
is established in 
Theorem~\ref{THminmaxICfnSpfnZZ} of this paper.
An entry ``Cor.'' at $(f,g)$ means that 
\eqref{minmaxGenZZ} holds for this $(f,g)$,
which is a corollary of a result indicated in the same row or column
(because of the inclusion relations \eqref{FnclassIncl}).
An entry ``CntEx'' at $(f,g)$ means that
there is a counterexample that 
denies \eqref{minmaxGenZZ} for this $(f,g)$.
\begin{equation} 
\label{fencYesNo}
\begin{array}{c|cccccc}
 f  \ \ \backslash \   -g  
  &  \mathcal{F} & \mathcal{G} &  \mathcal{L} & \mathcal{M} & \mathcal{S}
\\
\hline
\mathcal{F} &  \mbox{\rm CntEx } & \mbox{\rm CntEx } &  \mbox{\rm CntEx } &  \mbox{\rm CntEx } & {\rm Th.\ref{THminmaxICfnSpfnZZ}} 
\\
\mathcal{G} &  \mbox{\rm CntEx } & \mbox{\rm CntEx } & \mbox{\rm CntEx } &  \mbox{\rm CntEx } & \mbox{\rm Th.\ref{THminmaxICfnConjSpfnZZ}} 
\\
\mathcal{L} &  \mbox{\rm CntEx } & \mbox{\rm CntEx } & \mbox{\rm \cite[Th.8.21]{Mdcasiam}} & \mbox{\rm CntEx }& {\rm Cor.}
\\
\mathcal{M} &  \mbox{\rm CntEx } & \mbox{\rm CntEx } & \mbox{\rm CntEx }& \mbox{\rm \cite[Th.8.21]{Mdcasiam}} & {\rm Cor.}  
\\
\mathcal{S} & \mbox{\rm Th.\ref{THminmaxICfnSpfnZZ}} & 
       \mbox{\rm Th.\ref{THminmaxICfnConjSpfnZZ}}  & {\rm Cor.} & {\rm Cor.} & {\rm Cor.}  
\\
\hline
\end{array}
\end{equation}

The following two examples 
show that the min-max formula \eqref{minmaxGenZZ}
is not necessarily true for \Mnat-convex $f$ and \Lnat-concave $g$.
By the inclusion relations in \eqref{FnclassIncl}, they also
serve as counterexamples for all entries ``CntEx'' in \eqref{fencYesNo}.
In the following,
$\overline{f}$ denotes the convex envelope of $f$ and
$\overline{g}$ the concave envelope of $g$;
we have
$\overline{f}, \overline{g}: \RR\sp{2} \to \RR$ (finite-valued).

\begin{example}[{\cite[Example 5.6]{Mbonn09}}]  \rm \label{EXnorealfenc}
Let $f, g: \ZZ\sp{2} \to \ZZ$ be defined as
\[
f(x_{1},x_{2}) = |x_{1}+x_{2}-1|,
\qquad
g(x_{1},x_{2}) = 1- |x_{1}-x_{2}|.
\]
The function $f$ is \Mnat-convex and $g$ is \Lnat-concave (actually L-concave).
We have
$ \min\{ f - g \} = 0$,
whereas 
$\min\{ \overline{f} - \overline{g} \}=-1$.
The integral conjugate functions
are given by
\[
f\sp{\bullet}(p_{1}, p_{2})  =
   \left\{  \begin{array}{ll}
    p_{1}            &   ((p_{1}, p_{2}) \in S),      \\
   + \infty      &   (\mbox{otherwise}),  \\
                      \end{array}  \right.
\quad
 g\sp{\circ}(p_{1}, p_{2})  =
   \left\{  \begin{array}{ll}
    -1            &   ((p_{1}, p_{2}) \in T),      \\
   - \infty      &   (\mbox{otherwise})  \\
                      \end{array}  \right.
\]
with
$S=\{ (-1,-1), (0,0), (1,1) \}$ and
$T=\{ (-1,1), (0,0), (1,-1) \}$.
Hence
$\max\{ g\sp{\circ} - f\sp{\bullet} \}
 = g\sp{\circ}(0,0) - f\sp{\bullet}(0,0) = -1 - 0 = -1$.
Therefore,
\begin{equation}
  \begin{array}{ccccccc}
 \min\{ f - g \}
 & > & \min\{ \overline{f} - \overline{g} \}
 & = &   \max\{ \overline{g}\sp{\circ}
         - \overline{f}\sp{\bullet}  \}
 & = &  \max\{ g\sp{\circ} - f\sp{\bullet} \}  
\\
 (0) &  & (-1)  & & (-1)  && (-1) 
  \end{array} 
\label{minminFailEx}
\end{equation}
(cf., \eqref{minminZRgen}, \eqref{maxmaxZRgen} in Section \ref{SCproofminmax}).
Note that the min-max identity fails because of the integrality gap
in the minimization problem.
Finally we add that the function
\begin{equation*} 
h(x_1,x_2) = f(x_1,x_2) - g(x_1,x_2)
  = |x_{1}+x_{2}-1| - (  1- |x_{1}-x_{2}| )
\end{equation*}
is an integrally convex function, with
\[
h(0,0) = 1 - 1 = 0, \quad
h(1,0) = 0 - 0 = 0, \quad
h(0,1) = 0 - 0 = 0, \quad
h(1,1) = 1 - 1 = 0
\]
\finbox
\end{example}

\begin{example}[{\cite[Example 5.7]{Mbonn09}}]  \rm \label{EXnointfenc}
Let $f, g: \ZZ\sp{2} \to \ZZ$ be defined as
\[
f(x_{1},x_{2}) = \max(0,x_{1}+x_{2}),
\qquad
g(x_{1},x_{2}) = \min(x_{1},x_{2}).
\]
The function $f$ is \Mnat-convex and $g$ is \Lnat-concave (actually L-concave).
We have
$ \min\{ f - g \} 
= \min\{ \overline{f} - \overline{g} \}=0$.
The integral conjugate functions are given as
$f\sp{\bullet} = \delta_{S}$ and
$g\sp{\circ} = - \delta_{T}$
in terms of the indicator functions of
 $S=\{ (0,0), (1,1) \}$ and $T=\{ (1,0), (0,1) \}$.
Since $S \cap T = \emptyset$,
$g\sp{\circ} - f\sp{\bullet}$ is identically equal to $-\infty$,
whereas 
$\max\{ \overline{g}\sp{\circ} - \overline{f}\sp{\bullet} \} = 0$
since 
$\overline{f}\sp{\bullet} = \delta_{\overline{S}}$,
$\overline{g}\sp{\circ} = -\delta_{\overline{T}}$,
and
$\overline{S} \cap \overline{T} = \{ (1/2, 1/2) \}$,
where
$\overline{S}$ and $\overline{T}$ denote
the convex hulls of $S$ and $T$, respectively.
Therefore,
\begin{equation}
  \begin{array}{ccccccc}
 \min\{ f - g \}
 & = & \min\{ \overline{f} - \overline{g} \}
 & = &   \max\{ \overline{g}\sp{\circ}
         - \overline{f}\sp{\bullet}  \}
 & > &  \max\{ g\sp{\circ} - f\sp{\bullet} \} 
\\
 (0) &  &  (0) & & (0) && (-\infty)
  \end{array} 
\label{maxmaxFailEx}
\end{equation}
(cf., \eqref{minminZRgen}, \eqref{maxmaxZRgen} in Section \ref{SCproofminmax}).
Note that the min-max identity fails because of the integrality gap
in the maximization problem.
Finally we add that the function
$h(x_1,x_2)  = \max(0,x_{1}+x_{2}) - \min(x_{1},x_{2}) = \max(|x_1|, |x_2|)$
is integrally convex.
\finbox
\end{example}

\subsection{Proof of Theorem~\ref{THminmaxICfnSpfnZZ}}
\label{SCproofminmax}

The main theorem (Theorem~\ref{THminmaxICfnSpfnZZ}) is proved in this section.
The proof consists of four steps.
In Steps 1 to 3, we prove the min-max formula \eqref{minmaxICfnSpfnZZ}
under the assumption that the minimum in \eqref{minmaxICfnSpfnZZ}
is finite,
while Step 4 deals with the other case
where the maximum in \eqref{minmaxICfnSpfnZZ} is assumed to be finite.

\subsubsection{Step 1: weak duality}

We start with the generic form of the Fenchel duality:
\begin{equation} 
 \min \{ f(x) - g(x) \mid  x \in \ZZ\sp{n}  \} 
= \max \{ g\sp{\circ}(p) - f\sp{\bullet}(p) \mid  p \in \ZZ\sp{n} \} .
\label{minmaxGenZZ2} 
\end{equation} 
For any functions
$f: \ZZ\sp{n} \to \ZZ \cup \{ +\infty \}$
and
$g: \ZZ\sp{n} \to \ZZ \cup \{ -\infty \}$
and for any integer vectors $x$ and $p$,
we have the following obvious relations:
\begin{align} 
 f (x) - g(x) 
&= (\langle p, x \rangle - g(x) )  - ( \langle p, x \rangle - f(x) )
\nonumber \\ 
& \geq 
\min_{y \in \ZZ\sp{n}} ( \langle p, y \rangle - g(y) ) 
- \max_{y \in \ZZ\sp{n}}(\langle p, y \rangle - f(y))
\nonumber \\ & =
 g\sp{\circ}(p) - f\sp{\bullet}(p).
\label{weakdualZZ} 
\end{align} 
This implies ``\,$\min \geq \max$\,'' in \eqref{minmaxGenZZ2}, that is,
\begin{equation} 
 \min \{ f(x) - g(x) \mid  x \in \ZZ\sp{n}  \} 
\geq \max \{ g\sp{\circ}(p) - f\sp{\bullet}(p) \mid  p \in \ZZ\sp{n} \}.
\label{minmaxGenWeak} 
\end{equation} 
This is called the {\em weak duality},
whereas \eqref{minmaxGenZZ2} is the {\em strong duality}.

Since the functions $f$ and $g$ are integer-valued,
the minimum on the left-hand side of \eqref{minmaxGenZZ2} is 
either a (finite) integer or $-\infty$
when $\dom f \cap \dom g \neq \emptyset$.
Therefore, if the minimum 
is finite, then 
there exists a vector $x\sp{*} \in \ZZ\sp{n}$
that attains the minimum.

To prove strong duality \eqref{minmaxGenZZ2} from weak duality \eqref{minmaxGenWeak}, 
it suffices to show the  existence of an integer vector 
$p\sp{*}$
for which the inequality in \eqref{weakdualZZ}
is an equality for $x = x\sp{*}$.
Thus the proof of \eqref{minmaxGenZZ2}
is reduced to showing the existence of
$p\sp{*} \in \ZZ\sp{n}$
satisfying
\begin{equation} 
x\sp{*} \in 
\argmax_{y \in \ZZ\sp{n}} (\langle p\sp{*}, y \rangle - f(y) ) 
          \cap 
\argmin_{y \in \ZZ\sp{n}} ( \langle p\sp{*}, y \rangle - g(y) ) .
\label{argminmaxZ} 
\end{equation}

\subsubsection{Step 2: convex extension}
\label{SCconvext}

We continue to work with the generic form 
\eqref{minmaxGenZZ2} and consider 
the continuous relaxation of the dual
(maximization) problem, in which the variable $p$ may be a real vector.

To this end, we assume that 
$f$ is extensible to a convex function
$\overline{f}: \RR\sp{n} \to \RR \cup \{ +\infty \}$
and 
$g$ is extensible to a concave function
$\overline{g}: \RR\sp{n} \to \RR \cup \{ -\infty \}$.
For technical reasons, it is further assumed that 
$\overline{f}$ and $-\overline{g}$ are locally polyhedral convex functions
in the sense that they are polyhedral convex functions 
when restricted to any integral box.
(This technical condition is met when $f$ is an integrally convex function
and $g$ is a separable concave function.) 

The definitions of conjugate functions are adapted for real vectors $p$ as
\begin{align*}
f\sp{\bullet}(p) 
& = \max \{ \langle p, x \rangle  - f(x) \mid  x \in \ZZ\sp{n} \}
\qquad ( p\in \RR\sp{n}),
\\
g\sp{\circ}(p) & = \min \{ \langle p, x \rangle -g(x)  \mid   x \in \ZZ\sp{n} \}
\qquad ( p\in \RR\sp{n}),
\end{align*}
which are compatible with 
\eqref{conjvexZpZ2} and \eqref{conjcavZpZ2}
for integer vectors $p$.
Then we have:
\begin{align} 
 \min \{ f(x) - g(x)  \mid  x \in \ZZ\sp{n}  \} 
& \geq \min \{ \overline{f}(x) -  \overline{g}(x) \mid x \in \RR\sp{n} \} 
\label{minminZRgen} 
\\ 
& \phantom{AAAAAA}  ||
\nonumber \\
 \max \{  g\sp{\circ}(p) - f\sp{\bullet}(p) \mid  p \in \ZZ\sp{n} \}
&\leq  \max \{  g\sp{\circ}(p) - f\sp{\bullet}(p) \mid  p \in \RR\sp{n} \} ,
\label{maxmaxZRgen} 
\end{align} 
where the vertical equality 
``\,$||$\,'' 
connecting \eqref{minminZRgen} and \eqref{maxmaxZRgen}
is the Fenchel duality in convex analysis \cite{BL06,HL01,Roc70},
which holds if
$\overline{f}$ and $-\overline{g}$ are locally polyhedral convex functions
and the minimum is finite.

In general, we may have strict inequalities
in \eqref{minminZRgen} and  \eqref{maxmaxZRgen},
as is demonstrated by
\eqref{minminFailEx} in Example \ref{EXnorealfenc} 
and
\eqref{maxmaxFailEx} in Example \ref{EXnointfenc}, respectively.
The following lemma states that 
an equality does hold in \eqref{minminZRgen} 
in the setting of Theorem~\ref{THminmaxICfnSpfnZZ}.

\begin{lemma}  \label{LMminmaxZR}
For an integer-valued integrally convex function
$f: \ZZ\sp{n} \to \ZZ \cup \{ +\infty \}$
with $\dom f \ne \emptyset$
and an integer-valued separable concave function
$\Psi: \ZZ\sp{n} \to \ZZ \cup \{ -\infty \}$
with $\dom \Psi \ne \emptyset$,
we have
\begin{align} 
 \min \{ f(x) - \Psi(x) \mid  x \in \ZZ\sp{n}  \} 
= \min \{ \overline{f}(x) - \overline{\Psi}(x) \mid  x \in \RR\sp{n}  \} ,
\label{minminZRintcnv} 
\end{align} 
and hence
\begin{align} 
 \min \{ f(x) - \Psi(x) \mid  x \in \ZZ\sp{n}  \} 
= \max \{ \Psi\sp{\circ}(p) - f\sp{\bullet}(p) \mid  p \in \RR\sp{n} \},
\label{minmaxZRintcnv} 
\end{align} 
where the minimum on the left-hand side is assumed to be finite. 
\end{lemma}

\begin{proof}
Since $f$ is integrally convex 
and $\Psi$ is separable concave,
the difference $f - \Psi$ is convex extensible
and its convex extension is equal to $\overline{f} - \overline{\Psi}$
by Proposition \ref{PRcnvextICsep},
that is, 
$\overline{f - \Psi} = \overline{f} - \overline{\Psi}$.
Therefore,
\begin{align*} 
& \min \{ f(x) - \Psi(x) \mid  x \in \ZZ\sp{n}  \} 
= \min \{ (\overline{f - \Psi})(x) \mid  x \in \ZZ\sp{n}  \} 
\\ & = \min \{ (\overline{f - \Psi})(x) \mid  x \in \RR\sp{n}  \} 
= \min \{ \overline{f}(x) - \overline{\Psi}(x) \mid  x \in \RR\sp{n}  \} ,
\end{align*} 
which shows \eqref{minminZRintcnv}.
Then \eqref{minmaxZRintcnv} follows from 
the equality 
``\,$||$\,'' 
between \eqref{minminZRgen} and \eqref{maxmaxZRgen}.
\end{proof}

Let $x\sp{*} \in \ZZ\sp{n}$
be an optimal solution to the primal problem 
(i.e., a minimizer on the left-hand side of \eqref{minmaxZRintcnv}).
Let $\hat p \in \RR\sp{n}$
be an optimal solution to the dual problem 
(i.e., a maximizer on the right-hand side of \eqref{minmaxZRintcnv}),
which is guaranteed to exist by the Fenchel duality 
in convex analysis (for continuous variables).
Then we have
\begin{align} 
x\sp{*} \in 
\argmax_{y \in \ZZ\sp{n}} (\langle \hat p, y \rangle - f(y) )
          \cap 
\argmin_{y \in \ZZ\sp{n}} ( \langle \hat p, y \rangle - \Psi(y) ). 
\label{argminmaxR} 
\end{align}

As is well known, this condition can be expressed in terms of subdifferentials as follows.
First we have
\[
 x\sp{*} \in \argmax_{y \in \ZZ\sp{n}} (\langle p, y \rangle - f(y) )
\iff 
 p \in \subgR f(x\sp{*}),
\]
where the definition of $\subgR f(x)$ is given in \eqref{subgZRdef00}.  
On setting $\Phi(x) := -\Psi(x)$ we similarly have
\[
 x\sp{*} \in \argmin_{y \in \ZZ\sp{n}} ( \langle p, y \rangle - \Psi(y) ) 
\iff 
 x\sp{*} \in \argmax_{y \in \ZZ\sp{n}} ( \langle -p, y \rangle - \Phi(y) ) 
\iff 
 -p \in \subgR \Phi (x\sp{*}).
\]
Accordingly, \eqref{argminmaxR} can be rewritten as
\begin{align} 
 \hat p \in  \subgR f(x\sp{*}) \cap (- \subgR \Phi(x\sp{*})).
\label{subgrZR} 
\end{align} 
Since such $\hat p$ exists,
we have, in particular, that 
\begin{align} 
 \subgR f(x\sp{*}) \cap (-\subgR \Phi(x\sp{*}))
 \neq \emptyset .
\label{subgrZRnonemp} 
\end{align}

\subsubsection{Step 3: dual integrality}
\label{SCintdualopt}

By adding integrality requirement
to \eqref{subgrZR},
we obtain the condition 
\begin{align} 
 p\sp{*} \in 
 \subgR f(x\sp{*}) \cap (-\subgR \Phi(x\sp{*}))
 \cap \ZZ\sp{n} 
\label{subgrZZ} 
\end{align} 
for an integral dual optimal solution $p\sp{*}$.
Note that this is equivalent to (or rewriting of)
the optimality condition given in \eqref{argminmaxZ} 
(with $g=\Psi = -\Phi$).
Thus, our task of proving 
Theorem~\ref{THminmaxICfnSpfnZZ}
is reduced to showing 
\begin{align} 
 \subgR f(x\sp{*}) \cap (- \subgR \Phi(x\sp{*}))
 \cap \ZZ\sp{n} 
 \neq \emptyset .
\label{subgrZZnonemp} 
\end{align}

Since $\Phi$ ($=-\Psi$) is an integer-valued separable convex function
defined on $\ZZ\sp{n}$,
the subdifferential
$\subgR \Phi(x\sp{*})$ is an integral box.
Hence we have
\begin{equation} \label{PhisubgR}
-\subgR \Phi(x\sp{*}) 
= \{ p \in \RR\sp{n} \mid \alpha_{j} \leq p_{j} \leq \beta_{j} \ \ (j=1,2,\ldots,n) \}
\end{equation}
for some 
$\alpha \in (\ZZ \cup \{ -\infty \})\sp{n}$ and
$\beta \in (\ZZ \cup \{ +\infty \})\sp{n}$.

Let $B=- \subgR \Phi(x\sp{*})$.
Then 
(i)  $B$ is an integral box by \eqref{PhisubgR} and 
(ii) 
$\subgR f(x\sp{*}) \cap B \neq \emptyset$
by \eqref{subgrZRnonemp}.
We want to show that these conditions imply
$\subgR f(x\sp{*}) \cap B
 \cap \ZZ\sp{n} 
 \neq \emptyset$
in \eqref{subgrZZnonemp}.
The main technical result (Theorem~\ref{THICsubgrBox})
states that this is indeed the case,
completing the proof of Theorem~\ref{THminmaxICfnSpfnZZ}.
The proof of Theorem~\ref{THICsubgrBox} is given in the next section.

\subsubsection{Step 4: finiteness assumption}
\label{SCfntmaxfntmin}

It remains to show that the finiteness of the maximum in \eqref{minmaxICfnSpfnZZ}
implies the finiteness of the minimum in \eqref{minmaxICfnSpfnZZ}.

\begin{lemma}\label{LMfntmaxfntmin}
For an integer-valued integrally convex function
$f: \ZZ\sp{n} \to \ZZ \cup \{ +\infty \}$
and an integer-valued separable concave function
$\Psi: \ZZ\sp{n} \to \ZZ \cup \{ -\infty \}$,
if $\max \{ \Psi\sp{\circ}(p) - f\sp{\bullet}(p) \mid  p \in \ZZ\sp{n} \}$ 
is finite, then 
$\min \{ f(x) - \Psi(x) \mid  x \in \ZZ\sp{n}  \}$
is also finite.
\end{lemma}

\begin{proof}
Suppose that the maximum,
$\max \{ \Psi\sp{\circ} - f\sp{\bullet} \}$, 
 is finite.
Since 
$\Psi\sp{\circ}(p) - f\sp{\bullet}(p)$
is integer-valued, there exists an integer vector $p\sp{*}$
that attains the maximum.
By the weak duality \eqref{minmaxGenWeak}, this implies that 
the minimum,
$\min \{ f - \Psi  \}$,
 is finite or else $+\infty$.
To prove by contradiction, suppose that the minimum is $+\infty$,
which means 
$\dom f \cap \dom \Psi = \emptyset$.
By Proposition~\ref{PRcnvextICsep} (2), we then have
\[
 \overline{\dom f} \cap \overline{\dom \Psi} = \emptyset .
\]
By the separation theorem (in convex analysis), 
there exists a hyperplane separating 
$\overline{\dom f}$ and $\overline{\dom \Psi}$.
Since $\dom f$ and $\dom \Psi$ are integrally convex sets,
we can take an integer vector as the normal vector to define the 
separating hyperplane.
That is, there exist
$q \in \ZZ\sp{n}$ and $C \in \ZZ$ such that
\begin{align*}
 \langle q, x \rangle &\leq C \quad\;\;\qquad (x \in \dom f), 
\\
 \langle q, x \rangle &\geq C+1 \qquad (x \in \dom \Psi) .
\end{align*}
Then we have
\begin{align*}
 \Psi\sp{\circ}(q+p\sp{*}) 
 &= \min\{ \langle q+p\sp{*}, x \rangle - \Psi(x) \mid x \in \ZZ\sp{n} \} 
\nonumber \\
 &= \min\{ \langle q+p\sp{*}, x \rangle - \Psi(x) \mid x \in \dom \Psi \} 
\nonumber \\
 &\geq (C+1) + \min\{ \langle p\sp{*}, x \rangle - \Psi(x) \mid x \in \dom \Psi \} 
\nonumber \\
 &= (C+1) + \Psi\sp{\circ}(p\sp{*}), 
\\
 f\sp{\bullet}(q+p\sp{*})
 &= \max\{ \langle q+p\sp{*}, x \rangle - f(x) \mid x \in \ZZ\sp{n} \} 
\nonumber \\
 &= \max\{ \langle q+p\sp{*}, x \rangle - f(x) \mid x \in \dom f \}
\nonumber \\
 &\leq C + \max\{ \langle p\sp{*}, x \rangle - f(x) \mid x \in \dom f \} 
\nonumber \\
 &= C + f\sp{\bullet}(p\sp{*}),
\end{align*}
from which we obtain a contradiction
\[
 \Psi\sp{\circ}(q+p\sp{*}) - f\sp{\bullet}(q+p\sp{*}) 
 \geq \Psi\sp{\circ}(p\sp{*}) - f\sp{\bullet}(p\sp{*}) + 1
 > \Psi\sp{\circ}(p\sp{*}) - f\sp{\bullet}(p\sp{*}).
\]
Therefore, 
$\min \{ f - \Psi \}$
must be finite.
\end{proof}

\subsection{Connection to min-max theorems on bisubmodular functions}
\label{SCbisubbox}

Let $N = \{ 1,2,\ldots, n  \}$ and 
denote by $3\sp{N}$ the set of all pairs $(X,Y)$ of disjoint subsets $X, Y$ of $N$,
that is,
$3\sp{N} = \{ (X,Y)  \mid  X, Y \subseteq N, \  X \cap Y = \emptyset \}$.
A function $f: 3\sp{N} \to \RR$ is called {\em bisubmodular} if
\begin{align*} 
& f(X_{1}, Y_{1}) + f(X_{2}, Y_{2}) 
\\ & \geq
 f(X_{1} \cap X_{2}, Y_{1} \cap Y_{2}) +
 f((X_{1} \cup X_{2}) \setminus (Y_{1} \cup Y_{2}), 
  (Y_{1} \cup Y_{2}) \setminus  (X_{1} \cup X_{2}) )
\end{align*} 
holds for all 
$(X_{1}, Y_{1}), (X_{2}, Y_{2})  \in 3\sp{N}$. 
In the following we assume 
$f(\emptyset,\emptyset) =0$.
The associated {\em bisubmodular polyhedron} is defined by
\[
P(f) = \{  z \in \RR\sp{n} \mid  
  z(X) - z(Y) \leq f(X,Y) \ \ \mbox{for all\ } (X,Y) \in 3\sp{N} \},
\]
which, in turn, determines $f$ by
\begin{equation}
f(X,Y) = \max \{ z(X) - z(Y)  \mid z \in P(f) \} 
\qquad ((X,Y) \in 3\sp{N}).
\label{bisubfnfromXY}
\end{equation}
In this section we restrict ourselves to the case of integer-valued $f$,
for which $P(f)$ is an integral polyhedron.	
The reader is referred to \cite[Section 3.5(b)]{Fuj05book} and \cite{Fuj14bisubmdc}
for bisubmodular functions and polyhedra.

In a study of $b$-matching degree-sequence polyhedra,
Cunningham--Green-Kr{\'o}tki \cite{CG91degseq}
obtained a min-max formula for the maximum component sum
$z(N) = \sum_{i \in N} z_{i}$ of $z \in P(f)$ 
upper-bounded by a given vector $w$.

\begin{theorem}[{\cite[Theorem 4.6]{CG91degseq}}] \label{THcungreen}
Let $f: 3\sp{N} \to \ZZ$ be an integer-valued bisubmodular function 
with 
$f(\emptyset,\emptyset) =0$,
and $w \in \ZZ\sp{n}$.
If there exists $z \in P(f)$ with $z \leq w$, 
then
\begin{align} 
&\max\{ z(N) \mid  z \in P(f) \cap \ZZ\sp{n}, \  z \leq w  \}
\nonumber \\ & 
= \min\{ f(X,Y) + w(N \setminus X) + w(Y) 
  \mid  (X, Y) \in 3\sp{N} \}.
\label{minmaxCG}
\end{align} 
\vspace{-1.7\baselineskip}\\
\finbox
\end{theorem}

The min-max formula \eqref{minmaxCG}
can be extended to a box constraint (with both upper and lower bounds on $z$).
This extension is given in \eqref{FPminmax} below.
Although this formula is not explicit in 
Fujishige--Patkar \cite{FP94},
it can be derived without difficulty from the results of \cite{FP94};
see Remark~\ref{RMminmaxFPder}.

\begin{theorem}[\cite{FP94}] \label{THfujpat}
Let $f: 3\sp{N} \to \ZZ$ be an integer-valued bisubmodular function 
with $f(\emptyset,\emptyset) =0$,
and $\alpha$ and  $\beta$ be integer vectors with $\alpha \leq \beta$.
If there exists $z \in P(f)$ with $\alpha \leq z \leq \beta$,
then, for each $(A,B) \in 3\sp{N}$, we have
\begin{align}
&\max \{ z(A) - z(B)  \mid z \in P(f) \cap \ZZ\sp{n} , \alpha \leq z \leq \beta \}
\nonumber \\ & =
\min \{ f(X,Y) + \beta(A \setminus X) + \beta(Y \setminus B) - \alpha(B \setminus Y) 
    - \alpha(X \setminus A) 
     \mid (X,Y) \in 3\sp{N} \}.
\label{FPminmax}
\end{align}
\vspace{-1.7\baselineskip}\\
\finbox
\end{theorem}

Theorem~\ref{THfujpat} can be derived
from our main result (Theorem~\ref{THminmaxICfnSpfnZZ})
as follows.
Let $\hat f$ denote the convex extension of 
the given bisubmodular function $f: 3\sp{N} \to \ZZ$,
as defined by Qi \cite[Eqn (5)]{Qi88}
as a generalization of the Lov{\'a}sz extension of a submodular function.
This function 
$\hat f: \RR\sp{n} \to \RR$ 
is a positively homogeneous convex function and 
it is an extension of $f$ in the sense of 
$\hat f (e_{X}-e_{Y}) = f(X,Y)$
for all $(X,Y) \in 3\sp{N}$,
where, for any $Z \subseteq N$,
$e_{Z}$ denotes the characteristic vector of $Z$.
It follows from the positive homogeneity of $\hat f$ 
and Lemma~11 of \cite{Qi88} that
\[
 \frac{1}{2}\left(\hat f(x) + \hat f(y)\right) \geq 
 \hat f\left(\frac{x+y}{2}\right) 
\qquad (x, y \in \ZZ\sp{n}).
\]
This implies, by Theorem~\ref{THfavtarProp33},
that the function $\hat f$ restricted to $\ZZ\sp{n}$
is an integrally convex function.
In the following we denote
the restriction of $\hat f$ to $\ZZ\sp{n}$
also by $\hat f$.
This function is an integer-valued integrally convex function.

Fix $(A,B) \in 3\sp{N}$ and let
$C = N \setminus (A \cup B)$.
We define a separable concave function 
$\Psi(x) = \sum_{i \in N} \psi_{i}(x_{i})$ 
with $\psi_{i}: \ZZ \to \ZZ$ as follows:
For $i \in A$,
\begin{equation}\label{sepconcaveA}
 \psi_{i}(k) =
\begin{cases} 
 \alpha_{i}(k-1)  &  (k \geq 1),  
\\ 
 \beta_{i}(k-1) & (k \leq 1);
\end{cases}
\end{equation}
For $i \in B$,
\begin{equation}\label{sepconcaveB}
 \psi_{i}(k) =
\begin{cases} 
 \alpha_{i}(k+1)  &  (k \geq -1),  
\\ 
 \beta_{i}(k+1) & (k \leq -1);
\end{cases}
\end{equation}
For $i \in C$,
\begin{equation}\label{sepconcaveC}
 \psi_{i}(k) =
\begin{cases} 
 \alpha_{i} k  &  (k \geq 0),  
\\ 
 \beta_{i} k & (k \leq 0) .
\end{cases} 
\end{equation}

We intend to apply
Theorem~\ref{THminmaxICfnSpfnZZ}
to the integer-valued integrally convex function $\hat f$
and the integer-valued separable concave function $\Psi$.
For these functions the min-max formula \eqref{minmaxICfnSpfnZZ} reads
\begin{align} 
 \min \{ \hat f(x) - \Psi(x) \mid  x \in \ZZ\sp{n}  \} 
= \max \{ \Psi\sp{\circ}(p) - \hat f\sp{\bullet}(p) \mid  p \in \ZZ\sp{n} \} .
\label{minmaxBisub} 
\end{align}

On identifying a vector $x \in \{-1, 0, +1 \}\sp{n}$
with $(X,Y) \in 3\sp{N}$ by $x = e_{X}-e_{Y}$,
we have 
\[
 -\Psi(x) = \beta(A \setminus X) + \beta(Y \setminus B) - \alpha(B \setminus Y) 
    - \alpha(X \setminus A) ,
\]
which is easily verified from \eqref{sepconcaveA}--\eqref{sepconcaveC}.
Hence we have
\begin{equation}\label{RHS-FPminmax-fPsi}
 \min \{ \hat f(x) - \Psi(x) \mid x \in \{-1, 0, +1 \}\sp{n} \}
 = \mbox{RHS of \eqref{FPminmax}} .
\end{equation}
As is shown later, the minimization over $\{-1, 0, +1 \}\sp{n}$ 
here can be replaced by that over $\ZZ\sp{n}$, that is, 
\begin{equation}\label{min01=minZ}
 \min \{ \hat f(x) - \Psi(x) \mid x \in \{-1, 0, +1 \}\sp{n}\} 
  = \min \{ \hat f(x) - \Psi(x) \mid x \in \ZZ\sp{n}\}.
\end{equation}
Then, we 
\begin{equation}\label{RHS-FPminmax-fPsiZ}
 \mbox{LHS of \eqref{minmaxBisub}} 
 = \mbox{RHS of \eqref{FPminmax}} .
\end{equation}

Next we turn to the maximization on the right-hand side of \eqref{minmaxBisub}. 
The conjugate function 
$\hat f\sp{\bullet}$ is equal to 
the indicator function of $P(f)$.
That is,
$\hat f\sp{\bullet}(p)$
is equal to $0$ if $p \in P(f)$, and to $+\infty$ otherwise.
The concave conjugate
$\Psi\sp{\circ}$ 
is given by
\[
 \Psi\sp{\circ}(p) =
\begin{cases} 
 p(A)-p(B)  &  (\alpha \leq p \leq \beta),  
\\ 
 -\infty & (\mbox{\rm otherwise})
\end{cases} 
\]
for $p \in \ZZ\sp{n}$.
Indeed, for $i \in A$ we obtain
\[
 \psi\sp{\circ}_{i}(\ell) =
\begin{cases} 
 \ell  &  (\alpha_{i} \leq \ell \leq \beta_{i}),  
\\ 
 -\infty & (\mbox{\rm otherwise}) 
\end{cases}
\]
from \eqref{sepconcaveA};
for $i \in B$ we obtain
\[
 \psi\sp{\circ}_{i}(\ell) =
\begin{cases} 
 -\ell  &  (\alpha_{i} \leq \ell \leq \beta_{i}),  
\\ 
 -\infty & (\mbox{\rm otherwise}) 
\end{cases} 
\]
from \eqref{sepconcaveB};
for $i \in C$ we obtain
\[
 \psi\sp{\circ}_{i}(\ell) =
\begin{cases} 
 0  &  (\alpha_{i} \leq \ell \leq \beta_{i}),  
\\ 
 -\infty & (\mbox{\rm otherwise})
\end{cases} 
\]
from \eqref{sepconcaveC}.
Therefore, 
for the right-hand side of \eqref{minmaxBisub},
we have
\begin{align}
 & \mbox{RHS of \eqref{minmaxBisub}} 
= \max \{ \Psi\sp{\circ}(p) - \hat f\sp{\bullet}(p) \mid  p \in \ZZ\sp{n} \} 
\nonumber \\
 &= \max\{ p(A) - p(B) \mid p \in P(f) \cap \ZZ\sp{n}, \alpha \leq p \leq \beta\} 
 = \mbox{LHS of \eqref{FPminmax}} ,
\label{BsubminmaxICfnSpfn} 
\end{align}
where the variable $p$ corresponds to $z$ in \eqref{FPminmax}.
The combination of 
\eqref{BsubminmaxICfnSpfn}
with \eqref{minmaxBisub} and 
\eqref{RHS-FPminmax-fPsiZ}
implies the min-max relation in \eqref{FPminmax}.
It is added that the assumption of Theorem~\ref{THminmaxICfnSpfnZZ} is met,
since the value of \eqref{BsubminmaxICfnSpfn} is finite.

It remains to show \eqref{min01=minZ}.
Let $\hat x \in \ZZ\sp{n}$ be a minimizer on the right-hand side of \eqref{min01=minZ}
with $\| \hat x \|_{\infty}$ minimum.
To prove by contradiction, assume  
$\| \hat x \|_{\infty} \geq 2$.
Define
$U,V \subseteq N$ by 
\[
 U = \{ i \in N \mid \hat x_{i} = \| \hat x \|_{\infty}\}, \quad
 V = \{ i \in N \mid \hat x_{i} = -\| \hat x \|_{\infty}\},
\]
and let $d = e_{U} - e_{V}$.
By \eqref{sepconcaveA}--\eqref{sepconcaveC},
each $\psi_{i}$ is a linear (affine) function 
on each of the intervals $(-\infty,-1]$ and $[+1,+\infty)$.
Combining this with the fundamental property of 
the Lov{\'a}sz extension $\hat f$,
we see that there exists a vector $q \in \ZZ\sp{n}$ for which
\[
 (\hat f - \Psi)(\hat x \pm d) 
=  (\hat f - \Psi)(\hat x) \pm (f(U,V) + \langle q, d \rangle) 
\] 
holds, where the double-sign corresponds.
Since 
$\hat x$ is a minimizer of $\hat f - \Psi$,
we must have $f(U,V) + \langle q, d \rangle = 0$.
This implies, however, that
$\hat x - d$ is also a minimizer of $\hat f - \Psi$,
whereas we have $\| \hat x - d\|_{\infty} < \|\hat x\|_{\infty}$, a contradiction.
We have thus completed the derivation of 
Theorem~\ref{THfujpat} from Theorem~\ref{THminmaxICfnSpfnZZ}.

\begin{remark} \rm \label{RMminmaxFPder}
The min-max formula \eqref{FPminmax}
can be derived from the results of \cite{FP94} as follows.
Given $\alpha, \beta \in \ZZ\sp{n}$ with $\alpha \leq \beta$,
we can consider a bisubmodular function
$w_{\alpha \beta}$ defined by
\[
 w_{\alpha \beta}(X,Y) = \beta(X) - \alpha(Y)
\]
for disjoint subsets $X$ and $Y$.
The convolution of $f$ with $w = w_{\alpha \beta}$ is defined (and denoted) as 
\begin{align}
&
(f \circ w)(A,B) 
\nonumber \\ & =
\min \{ f(X,Y) + w(A \setminus X, B \setminus Y) + w(Y \setminus B,X \setminus A) 
      \mid (X,Y) \in 3\sp{N} \}
\nonumber \\ & =
\min \{ f(X,Y) + \beta(A \setminus X) + \beta(Y \setminus B) 
   - \alpha(B \setminus Y) - \alpha(X \setminus A) 
     \mid (X,Y) \in 3\sp{N} \}.
\label{fwdef}
\end{align}
This function is bisubmodular \cite[Theorem 3.2]{FP94}.
By \eqref{bisubfnfromXY} applied to $f \circ w$,
we obtain 
\begin{equation}
(f \circ w)(A,B) = \max \{ z(A) - z(B)  \mid z \in P(f \circ w) \} 
\qquad ((A,B) \in 3\sp{N}).
\label{convfw}
\end{equation}
On the other hand, Theorem 3.3 of \cite{FP94} shows
\begin{equation}
P(f \circ w) = P(f) \cap P(w) 
= \{  z \mid z \in P(f), \alpha \leq z \leq \beta \} .
\label{PfwPfPw}
\end{equation}
By substituting this expression 
into $P(f \circ w)$ on the right-hand side of \eqref{convfw}
we obtain
\begin{align}
(f \circ w)(A,B) 
&= 
\max \{ z(A) - z(B)  \mid z \in P(f \circ w) \} 
\nonumber \\
&=\max \{ z(A) - z(B)  \mid z \in P(f) \cap P(w) \} 
\nonumber \\
&=\max \{ z(A) - z(B)  \mid z \in P(f), \alpha \leq z \leq \beta \}.
\label{fwXYmax}
\end{align}
The combination of \eqref{fwdef} and \eqref{fwXYmax}
gives the desired equality \eqref{FPminmax}.

Finally we mention that the paper \cite{FP94}
considers a more general setting where 
$f$ is a bisubmodular function 
defined on a subset $\mathcal{F}$ of $3\sp{N}$ such that 
\begin{align*} 
& (X_{1}, Y_{1}), (X_{2}, Y_{2})  \in \mathcal{F} 
\ \Longrightarrow \ 
 (X_{1} \cap X_{2}, Y_{1} \cap Y_{2}) \in \mathcal{F},
\\ &
 (X_{1}, Y_{1}), (X_{2}, Y_{2})  \in \mathcal{F} 
\ \Longrightarrow \ 
 ((X_{1} \cup X_{2}) \setminus (Y_{1} \cup Y_{2}), 
  (Y_{1} \cup Y_{2}) \setminus (X_{1} \cup X_{2}) )
\in \mathcal{F}.
\end{align*} 
The min-max formula \eqref{FPminmax} remains true in this general case.
\finbox 
\end{remark}


\section{Integral Subgradients}
\label{SCsubrZ}

\subsection{Results}
\label{SCresult}

In this section we are interested in integral subgradients of
an integer-valued integrally convex function
$f: \ZZ\sp{n} \to \ZZ \cup \{ +\infty \}$.
Recall the definition of
the subdifferential 
\begin{equation} \label{subgZRdef0}
 \subgR f(x)
= \{ p \in  \RR\sp{n} \mid    
  f(y) - f(x)  \geq  \langle p, y - x \rangle   \ \ \mbox{\rm for all }  y \in \ZZ\sp{n} \}
\end{equation}
at $x \in \dom f$.
An integer vector $p$ belonging to $\subgR f(x)$, if any,
is called an integral subgradient of $f$ at $x$.

Integral subdifferentiability
of integer-valued integrally convex functions
is recently established by
Murota--Tamura \cite{MT20subgrIC}.

\begin{theorem}[\cite{MT20subgrIC}]  \label{THsubgrIC}
Let $f: \ZZ\sp{n} \to \ZZ \cup \{ +\infty \}$
be an integer-valued integrally convex function.
For every $x \in \dom f$,
we have
$\subgR f(x) \cap \ZZ\sp{n} \neq \emptyset$.
\finbox
\end{theorem}

The main technical result of the present paper
(Theorem~\ref{THICsubgrBox})
is a strengthening of this theorem
with an additional box condition.
Recall that an integral box means a set $B$ of real vectors represented as
\[
  B = \{ p \in \RR\sp{n} \mid \alpha \leq p \leq \beta \}
\]
with two integer vectors 
$\alpha \in (\ZZ \cup \{ -\infty \})\sp{n}$ and 
$\beta \in (\ZZ \cup \{ +\infty \})\sp{n}$ satisfying $\alpha \leq \beta$.
For convenience, we present the theorem again.
The proof is given in Sections \ref{SCfoumotzP}--\ref{SCfmPBproof}.

\medskip
\noindent
{\bf Theorem~\ref{THICsubgrBox}.} (Main technical result, again) \ 
\textit{
Let $f: \ZZ\sp{n} \to \ZZ \cup \{ +\infty \}$
be an integer-valued integrally convex function, $x \in \dom f$,
and $B$ be an integral box.
If $\subgR f(x) \cap B$ is nonempty,
then $\subgR f(x) \cap B$ 
is a polyhedron containing an integer vector.
If, in addition, $\subgR f(x) \cap B$ is bounded,
then $\subgR f(x) \cap B$ has an integral vertex.
}
\finbox

\medskip

As an immediate corollary, we can obtain the following statement.

\begin{corollary}  \label{COsubgrICp}
Let $f: \ZZ\sp{n} \to \ZZ \cup \{ +\infty \}$
be an integer-valued integrally convex function, $x \in \dom f$,
and $p \in \subgR f(x)$.
Then there exists 
an integer vector $q \in \subgR f(x)$
satisfying 
$\lfloor p_{i} \rfloor \leq  q_{i} \leq \lceil p_{i} \rceil$  $(i=1,2,\ldots, n)$.
\end{corollary}
\begin{proof}
Let $B=[\alpha,\beta]_{\RR}$ be the integral box defined by 
$\alpha_{i}=\lfloor p_{i} \rfloor$
and $\beta_{i} = \lceil p_{i} \rceil$ for $i=1,2,\ldots, n$.
Then $p \in \subgR f(x) \cap B$, which shows
$\subgR f(x) \cap B \neq \emptyset$.
By Theorem~\ref{THICsubgrBox}, there exists 
$q \in \subgR f(x) \cap B \cap \ZZ\sp{n}$, 
which is an integer vector $q$ in $\subgR f(x)$
satisfying 
$\lfloor p_{i} \rfloor \leq  q_{i} \leq \lceil p_{i} \rceil$  $(i=1,2,\ldots, n)$.
\end{proof}


In the following we make supplementary remarks
concerning the contents of Theorems \ref{THICsubgrBox} and \ref{THsubgrIC}.

\begin{remark} \rm \label{RMsubg}
Integral subdifferentiability is not guaranteed
without the assumption of integral convexity,
as the following example
\cite[Example 1.1]{Mdca98} shows.

Let 
$D = \{ (0,0,0), \pm (1,1,0), \pm (0,1,1), \pm (1,0,1) \}$
and $f: \mathbb{Z}^3 \to \mathbb{Z} \cup \{+\infty\}$ be defined by
\begin{align*}
f(x_{1},x_{2},x_{3}) =
 \begin{cases} (x_{1}+x_{2}+x_{3})/2 & (x \in D), \\
                +\infty & (\textrm{otherwise}).  
  \end{cases}
\end{align*}
This function can be naturally extended to a convex function on 
the convex hull $\overline{D}$ of $D$
and $D$ is hole-free in the sense of
$D =  \overline{D} \cap \mathbb{Z}^{3}$.
However, $D$ is not integrally convex since
for $x=(1,1,0)$ and $y=(-1,0,-1)$ we have 
$(x+y)/2 = (0, 1/2, -1/2)$,
$N((0, 1/2, -1/2)) = \{ (0,0,0), (0,1,0), (0,0,-1), (0,1,-1)  \}$,
and hence
$N((0, 1/2, -1/2)) \cap D    \allowbreak   = \{ (0,0,0) \}$.
Therefore, $f$ is not integrally convex. 

The subdifferential of $f$ at the origin $\veczero=(0,0,0)$
consists of real vectors $p$ that satisfy the inequality
$f(y) - f(\veczero) \ge \langle p, y \rangle$ for all $y \in D$.
This condition is given by
\begin{center}
\begin{tabular}{ccc}
$ 1 \ge  p_{1} + p_{2}$, & $ 1 \ge  p_{2} + p_{3}$, & $ 1 \ge  p_{3} + p_{1}$, \\
$-1 \ge -p_{1} - p_{2}$, & $-1 \ge -p_{2} - p_{3}$, & $-1 \ge -p_{3} - p_{1}$.
\end{tabular}
\end{center}
This system of inequalities has a unique solution 
$(p_{1}, p_{2}, p_{3}) = (1/2, 1/2, 1/2)$,
which means that
$\subgR f(\veczero) = \{ (1/2, 1/2, 1/2) \}$
and 
$\subgR f(\veczero) \cap \ZZ\sp{3} = \emptyset$.
\finbox
\end{remark}

\begin{remark} \rm \label{RMsubgNotIntPolyh}
Theorem~\ref{THsubgrIC} states that 
$\subgR f(x) \cap \ZZ\sp{n} \neq \emptyset$,
but it does not claim a stronger statement that $\subgR f(x)$ is an integer polyhedron.
Indeed, $\subgR f(x)$ is not necessarily an integer polyhedron, 
as the following example \cite[Remark 3.1]{MT20subgrIC} shows.

Let 
$f: \mathbb{Z}^3 \to \mathbb{Z} \cup \{+\infty\}$ be defined by
$f(0,0,0)=0$ and $f(1,1,0)=f(0,1,1)=f(1,0,1)=1$,
with $\dom f = \{ (0,0,0), (1,1,0), (0,1,1), (1,0,1) \}$.
This function is integrally convex
 and the subdifferential of $f$ at the origin is given as
\[
\subgR f(\veczero) = \{ p \in \RR\sp{3} \mid 
 p_{1} + p_{2} \leq 1,  
 p_{2} + p_{3} \leq 1,  
 p_{1} + p_{3} \leq 1   \},
\]
which is not an integer polyhedron, having a non-integral vertex at $p=(1/2, 1/2, 1/2)$.
Nevertheless, 
$\subgR f(\veczero) \cap B$
contains an integer vector 
for every integral box $B$,
which is the claim of Theorem~\ref{THICsubgrBox}.
\finbox
\end{remark}

\begin{remark} \rm \label{BoundedSG}
It is pointed out in \cite[Remark 4.1]{MT20subgrIC}
that, if $\subgR f(x)$ is a bounded polyhedron for 
an integer-valued integrally convex function $f$,
then $\subgR f(x)$ has an integral vertex,
although not every vertex of $\subgR f(x)$ is integral.
A concrete example is given here.

Let 
$ D = \{ x \in \{ -1,0,+1 \}\sp{3}  \mid  |x_{1}| + |x_{2}| + |x_{3}| \leq 2 \}$,
which is an integrally convex set.
Define $f$ on $D$ by
$f(\veczero) = 0$ and
$f(x) = 1$   $(x \in D \setminus \{ \veczero \})$.
This $f$ is an integer-valued integrally convex function
and $\subgR f(\veczero)$ is a bounded polyhedron described by the following inequalities:
\begin{align*}
& p_1 \pm p_2          \leq 1, 
\quad
 -p_1 \pm p_2          \leq 1,
\quad
  p_1         \pm p_3  \leq 1,
\quad
 -p_1         \pm p_3  \leq 1,
\\
&          p_2 \pm p_3  \leq 1,
\quad
         -p_2 \pm p_3  \leq 1,
\qquad 
\pm p_1         \leq 1,
\quad
      \pm p_2          \leq 1,
\quad
              \pm p_3  \leq 1 .
\end{align*}
The polyhedron $\subgR f(\veczero)$ has eight non-integral vertices
$(\pm 1/2, \pm 1/2, \pm 1/2)$
(with arbitrary combinations of double signs)
and six integral vertices
$(\pm 1, 0, 0)$,
$(0, \pm 1, 0)$, and
$(0,0, \pm 1)$.
\finbox
\end{remark}

It is in order here to briefly mention how 
Theorem~\ref{THICsubgrBox} is proved
in the remainder of this section.
Let $P := \subgR f(x)$ for notational simplicity,
and for each $\ell = 1,2,\ldots,n$,
let $[P]_{\ell}$ denote
the projection of $P$
to the space of $(p_{\ell},p_{\ell+1},\ldots,p_n)$.
Since $P$ is a polyhedron,
each $[P]_{\ell}$ is also a polyhedron.
In \cite{MT20subgrIC}, a 
hierarchical system of inequalities 
to describe $[P]_{\ell}$ for $\ell = 1,2,\ldots,n$
(Theorem~\ref{THelimICProj} in Section \ref{SCfoumotzP})
was derived by means of the Fourier--Motzkin elimination
and then Theorem~\ref{THsubgrIC} was proved as a fairly easy consequence 
of this description.
We extend this approach to prove Theorem~\ref{THICsubgrBox}.
Namely, in Theorem~\ref{THelimICboxProj} in Section \ref{SCfmPBgen},
we derive a hierarchical system of inequalities 
to describe 
the projection 
$[P \cap B]_{\ell}$
of $P \cap B$
to the space of $(p_{\ell},p_{\ell+1},\ldots,p_n)$
for $\ell = 1,2,\ldots,n$.
Here again we rely on the Fourier--Motzkin elimination.

\subsection{Fourier--Motzkin elimination for the subdifferential $\subgR f(x)$}
\label{SCfoumotzP}

In this section we briefly
describe the approach of 
\cite{MT20subgrIC}
to prove Theorem~\ref{THsubgrIC} 
(integral subdifferentiability)
with the aid of the Fourier--Motzkin elimination.
In so doing we intend to clarify the geometrical essence 
of the argument in \cite{MT20subgrIC}
in a form convenient for the proof of 
Theorem~\ref{THICsubgrBox}
(integral subdifferentiability with a box).
Recall our notation $P$ for the subdifferential $\subgR f(x)$
(for a fixed $x \in \dom f$)
and $[P]_{\ell}$ 
for the projection of $P$
to the space of $(p_{\ell},p_{\ell+1},\ldots,p_n)$,
where $\ell = 1,2,\ldots,n$.

To obtain an expression of $\subgR f(x)$,
we can make use of Theorem~\ref{THintcnvlocopt}
for the minimality of an integrally convex function.
Namely, in the definition of  $\subgR f(x)$ in (\ref{subgZRdef0}),
it suffices to consider  $y=x + d$ with $d \in \{-1,0,+1\}^n$,
and then
\begin{equation} \label{subgZR0def}
 \subgR f(x)
= \{ p \in  \RR\sp{n} \mid    
  \sum_{j=1}\sp{n} d_{j} p_{j} \leq f(x+d) - f(x) 
 \ \ \mbox{for all} \ \  d \in \{ -1,0,+1 \}\sp{n} \} .
\end{equation}
We represent the system of inequalities 
$\sum_{j=1}\sp{n} d_{j} p_{j} \leq f(x+d) - f(x) $ 
for $d$ with $f(x+d) < +\infty$
in a matrix form as
\begin{equation}\label{ineqApb}
 A p \leq b.
\end{equation}
Let $I$ denote the row set of $A$ and 
\[
A = ( a_{ij} \mid i \in I, j \in \{ 1,2,\ldots, n \}).
\]
We denote the $i$th row vector of $A$ by $a_{i}$ for $i \in I$.
The row set $I$ is 
indexed by $d \in \{ -1,0,+1 \}\sp{n}$ with $f(x+d) < +\infty$,
and $a_{i}$ is equal to the corresponding $d$ for $i \in I$;
we have $a_{ij} = d_{j}$ for $j=1,2,\ldots, n$ and 
\[
b_{i}= f(x+a_{i})-f(x).
\]
Note that $a_{ij} \in \{ -1,0,+1 \}$
and $a_{i} \in \{ -1,0,+1 \}\sp{n}$ for all $i$ and $j$.

An inequality system to describe the projections
$[P]_{\ell}$ for $\ell = 1,2,\ldots,n$
can be obtained 
by applying the Fourier--Motzkin elimination procedure \cite{Sch86} to 
the system of inequalities (\ref{ineqApb}),
where
the variable $p_{1}$ is eliminated first, and then $p_{2}, p_{3}, \ldots$,
to finally obtain an inequality in $p_{n}$ only.

By virtue of the integral convexity of $f$, a drastic simplification occurs
in this elimination process. 
The inequalities that are generated  
are actually redundant and need not be added
to the current system of inequalities,
which is a crucial observation made in \cite{MT20subgrIC}.
It is shown in \cite{MT20subgrIC}
that the Fourier--Motzkin elimination procedure
applied to \eqref{ineqApb} 
results in a system of inequalities
that is equivalent to \eqref{FMp12n} below:
\begin{align}
 \max_{k \in \hat I_{1}^{-}} \left\{ \sum_{j=2}^n a_{kj}p_{j} - b_{k} \right\} 
  \leq & \ p_{1} \leq  
 \min_{i \in \hat I_{1}^{+}} \left\{ b_{i} - \sum_{j=2}^n a_{ij}p_{j}  \right\} ,
\nonumber \\
 \max_{k \in \hat I_{2}^{-}} \left\{ \sum_{j=3}^n a_{kj}p_{j} - b_{k} \right\} 
  \leq & \  p_{2} \leq  
 \min_{i \in \hat I_{2}^{+}} \left\{ b_{i} - \sum_{j=3}^n a_{ij}p_{j}  \right\} ,
\nonumber \\
  & \ \vdots 
\label{FMp12n}
\\
 \max_{k \in \hat I_{n-1}^{-}} \left\{ a_{kn}p_{n} - b_{k} \right\} 
  \leq & \  p_{n-1} \leq  
 \min_{i \in \hat I_{n-1}^{+}} \left\{ b_{i} - a_{in}p_{n}  \right\} ,
\nonumber \\
 \max_{k \in \hat I_{n}^{-}} \left\{ - b_{k} \right\} 
  \leq & \ p_{n} \leq  
 \min_{i \in \hat I_{n}^{+}} \left\{ b_{i} \right\} .
\nonumber 
\end{align}
Here the index sets are defined as follows:
Let $\hat I_{0}^{0} = I$ and for  $j=1, 2 ,\ldots,n$, define
\begin{align}
 \hat I_{j}^{+} &= \{ i \in \hat I_{j-1}^{0} \mid a_{ij} = +1 \},   
\nonumber \\
 \hat I_{j}^{-} &= \{ i \in \hat I_{j-1}^{0} \mid a_{ij} = -1 \},
\label{indIjdef} \\
 \hat I_{j}^{0} &= \{ i \in \hat I_{j-1}^{0} \mid a_{ij} = 0 \},   
\nonumber 
\end{align}
where $\hat I_{j}^{+}$ and/or $\hat I_{j}^{-}$ may possibly be empty. 
This result can be rephrased in terms of projections as follows.

\begin{theorem}[\cite{MT20subgrIC}] \label{THelimICProj}
Let $f: \ZZ\sp{n} \to \ZZ \cup \{ +\infty \}$
be an integer-valued integrally convex function and $x \in \dom f$.
For each $\ell = 1,2,\ldots,n$,
the projection of $\subgR f(x)$
to the space of $(p_{\ell},p_{\ell+1},\ldots,p_n)$
is described by the last $n-\ell+1$ inequalities 
in \eqref{FMp12n} for 
$p_{\ell},p_{\ell+1},\ldots,p_n$.
\finbox
\end{theorem}

\subsubsection*{Proof of Theorem~\ref{THsubgrIC} by Theorem~\ref{THelimICProj}}

Since $\subgR f(x)$ is nonempty,
there exists a real vector $p$ satisfying all the inequalities in (\ref{FMp12n}).
It is even true that, 
for any choice of 
$(p_{\ell},p_{\ell+1},\ldots,p_n)$
satisfying the last $n-\ell+1$ inequalities 
of \eqref{FMp12n},
the ($\ell-1$)-st inequality of \eqref{FMp12n}
prescribes a nonempty interval for a possible choice of $p_{\ell-1}$.
As for integrality, the last inequality in (\ref{FMp12n}) shows
that we can choose an integral $p_{n} \in \ZZ$,
since $b_{i} \in \ZZ$ 
for $i \in \hat I_{n}^{-} \cup \hat I_{n}^{+}$.
Then the next-to-last inequality shows that
we can choose an integral $p_{n-1} \in \ZZ$,
since
$a_{kn}p_{n} - b_{k} \in \ZZ$ for $k \in \hat I_{n-1}^{-}$
and
$b_{i} - a_{in}p_{n} \in \ZZ$ for $i \in \hat I_{n-1}^{+}$.
Continuing in this way we can see the existence of an integer vector $p \in \ZZ\sp{n}$
satisfying (\ref{FMp12n}).  
This shows $\subgR f(x) \cap \ZZ\sp{n} \neq \emptyset$,
which completes the proof of Theorem~\ref{THsubgrIC}.

\begin{remark} \rm \label{RMsystemPB}
It follows immediately from Theorem~\ref{THelimICProj} that 
$P \cap B$  ($= \subgR f(x) \cap [\alpha,\beta]_{\RR}$)
is described by the following inequalities:
\begin{align}
\max \left\{
 \max_{k \in \hat I_{1}^{-}} \left\{ \sum_{j=2}^n a_{kj}p_{j} - b_{k} \right\} \!\! ,
\alpha_{1}
 \right\}
  \leq & \ p_{1} \leq  
\min \left\{
 \min_{i \in \hat I_{1}^{+}} \left\{ b_{i} - \sum_{j=2}^n a_{ij}p_{j}  \right\} \!\! ,
\beta_{1}
 \right\} \! ,
\nonumber \\
\max \left\{
 \max_{k \in \hat I_{2}^{-}} \left\{ \sum_{j=3}^n a_{kj}p_{j} - b_{k} \right\} \!\!, 
\alpha_{2}
 \right\}
  \leq & \  p_{2} \leq  
\min \left\{
 \min_{i \in \hat I_{2}^{+}} \left\{ b_{i} - \sum_{j=3}^n a_{ij}p_{j}  \right\} \!\! ,
\beta_{2}
 \right\} \! ,
\nonumber \\
  & \ \vdots 
\label{FMconfunc}
\\
\max \! \left\{ \! 
 \max_{k \in \hat I_{n-1}^{-}} \left\{ a_{kn}p_{n} - b_{k} \right\} \!\! , 
\alpha_{n-1} \! 
 \right\}
  \leq & \  p_{n-1} \leq  
\min  \! \left\{ \! 
 \min_{i \in \hat I_{n-1}^{+}} \left\{ b_{i} -  a_{in}p_{n}  \right\} \!\! ,
\beta_{n-1} \! 
 \right\} \! ,
\nonumber \\
\max \left\{
 \max_{k \in \hat I_{n}^{-}} \left\{ - b_{k} \right\}, 
\alpha_{n}
 \right\}
  \leq & \  p_{n} \leq  
\min \left\{
 \min_{i \in \hat I_{n}^{+}} \left\{ b_{i}  \right\} ,
\beta_{n}
 \right\} ,
\nonumber 
\end{align}
where the index sets $\hat I_{j}^{+}$ and $\hat I_{j}^{-}$  
are defined in \eqref{indIjdef}.
It is certainly true that $p$ belongs to $P \cap B$
if and only if $p$ satisfies all inequalities in \eqref{FMconfunc}.
But it is not true that the projection 
$[P \cap B]_{\ell}$
of $P \cap B$
to the space of $(p_{\ell},p_{\ell+1},\ldots,p_n)$
is described by the last $n-\ell+1$ inequalities.
(This is because
$[P \cap B]_{\ell}$ may differ from $[P]_{\ell} \cap [B]_{\ell}$,
with $[B]_{\ell}$ denoting the projection of $B$.)
In particular, it may not be true that the last inequality
\[
\max \left\{
 \max_{k \in \hat I_{n}^{-}} \left\{ - b_{k} \right\}, 
\alpha_{n}
 \right\}
  \leq  \  p_{n} \leq  
\min \left\{
 \min_{i \in \hat I_{n}^{+}} \left\{ b_{i}  \right\} ,
\beta_{n}
 \right\}
\]
describes the projection of 
$P \cap B$ to the $p_{n}$-axis,
which is demonstrated by a simple example in Example~\ref{EXicbox2dim} below.
An important consequence of this phenomenon is that we
are not allowed to prove Theorem~\ref{THICsubgrBox}
on the basis of \eqref{FMconfunc} 
by extending the proof of Theorem~\ref{THsubgrIC}
based on Theorem~\ref{THelimICProj} (described above).
This is the reason for our (rather long) proof of 
Theorem~\ref{THICsubgrBox} given  
in Sections \ref{SCfmPBgen} and \ref{SCfmPBproof}.
\finbox
\end{remark}

\begin{example} \rm  \label{EXicbox2dim}
Let $f: \ZZ\sp{2} \to \ZZ \cup \{ +\infty \}$
be a function defined on $\{ -1,0, +1 \}\sp{2}$ by
\[
\begin{array}{lll}
 f(-1,1)=2, & f(0,1)=3, & f(1,1)=4, \\
 f(-1,0)=2, & f(0,0)=0, & f(1,0)=4, \\
 f(-1,-1)=3, & f(0,-1)=3, & f(1,-1)=4, \\
\end{array}
\]
which is integer-valued and integrally convex.
The subdifferential $P =  \subgR f(0,0)$ is described by
the following inequalities:
\begin{align}
& 
p_{1} + p_{2} \leq 4, \quad
p_{1} - p_{2} \leq 4, \quad
- p_{1} + p_{2} \leq 2, \quad
- p_{1} - p_{2} \leq 3, 
\nonumber
\\ & 
p_{1}  \leq 4, \quad
-p_{1}  \leq 2, \quad
p_{2}  \leq 3, \quad
-p_{2}  \leq 3,
\label{ineqPicbox2dim}
\end{align}
for which the Fourier--Motzkin elimination results in
\begin{align} 
 \max\{ p_{2}-2, -p_{2}-3, -2 \} \leq p_{1} \leq  \min \{ - p_{2}+4,   p_{2}+4, 4   \} , 
\quad
 -3 \leq p_{2} \leq 3.
\label{PfmelimD2}
\end{align}
Let $B$ be an integral box described by
\begin{equation}
  2 \leq p_{1}, \quad -4 \leq p_{2} \leq 4 .
\label{BineqD2}
\end{equation}
A system of inequalities describing $P \cap B$
is obtained by combining \eqref{PfmelimD2} and \eqref{BineqD2} as 
\begin{align*} 
 \max\{ p_{2}-2, -p_{2}-3 ,-2, 2 \} \leq & \ p_{1} 
  \leq  \min \{ - p_{2}+4,   p_{2}+4 ,4, +\infty  \} , 
\\
  -3 = \max \{ -3, -4 \} \leq & \ p_{2} \leq \min\{ 3, 4 \} = 3,     
\end{align*}
which corresponds to \eqref{FMconfunc}.
However, the last inequality $-3  \leq  p_{2} \leq 3$ here 
does not describe the projection
of $P \cap B$ to the $p_{2}$-axis,
which, actually, is the interval 
$-2  \leq  p_{2} \leq 2$
as Fig.~\ref{FGfmbox} shows.
\finbox
\end{example}

\begin{figure}
\centering
\includegraphics[width=0.4\textwidth,clip]{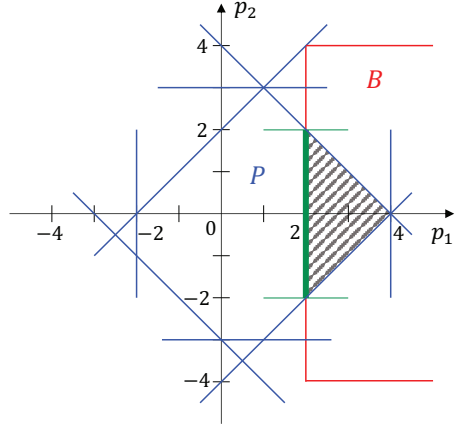}
\caption{$P \cap B$ in Example \ref{EXicbox2dim}} 
\label{FGfmbox}
\end{figure}


\subsection{Fourier--Motzkin elimination for $\subgR f(x) \cap B$}
\label{SCfmPBgen}

Recall notations
$P=\subgR f(x)$ and
\begin{equation} \label{BoxIneq1}
B =  \{ p \in \RR\sp{n} \mid \alpha_{j} \leq p_{j} \leq \beta_{j} \ (j=1,2,\ldots,n) \},
\end{equation}
where the possibility of 
$\alpha_{j} = -\infty$ and/or $\beta_{j}= +\infty$
is allowed.
In this section we are concerned with a system of
inequalities that describe  
the projections $[P \cap B]_{\ell}$ for $\ell = 1,2,\ldots,n$.
The result is stated in Theorem~\ref{THelimICboxProj},
with which Theorem~\ref{THICsubgrBox}
(the main technical result of this paper) is proved.

Recall that $P$ is described by the system of inequalities
$A p \leq b$ in \eqref{ineqApb}.
For each 
$\ell =1,2, \ldots,n$,
we look at those inequalities 
which contain the variable $p_{\ell}$.
Using the sets of indices defined as
\begin{align}
& I_{\ell}^{+} = \{ i \in I \mid a_{i \ell} = +1 \},   
\quad  
 I_{\ell}^{-} = \{ i \in I \mid a_{i \ell} = -1 \},   
\quad
 I_{\ell}^{0} = \{ i \in I \mid a_{i \ell} = 0 \},   
\nonumber
\\ &
 J_{\ell i}^{+} = \{ j \mid 1 \leq j \leq \ell -1, \ a_{ij} = +1 \}
\qquad  (i \in I),
\label{indIJelldef}
\\ &
 J_{\ell i}^{-} = \{ j \mid 1 \leq j \leq \ell -1, \  a_{ij} = -1 \}   
\qquad  (i \in I),
\nonumber
\end{align}
the inequalities of $A p \leq b$ are classified as
\begin{align}
 \sum_{j \in J_{\ell i}\sp{+} } p_{j} -  \sum_{j \in J_{\ell i}\sp{-} } p_{j} + p_{\ell} 
  + \sum_{j=\ell+1}\sp{n} a_{ij}p_{j} & \leq b_{i}
\qquad (i \in I_{\ell}\sp{+}) ,
\label{Axbpl+}
\\
 \sum_{j \in J_{\ell k}\sp{+} } p_{j} -  \sum_{j \in J_{\ell k}\sp{-} } p_{j} - p_{\ell} 
  + \sum_{j=\ell+1}\sp{n} a_{kj}p_{j} & \leq b_{k}
\qquad (k \in I_{\ell}\sp{-}), 
\label{Axbpl-}
\\
 \sum_{j \in J_{\ell h}\sp{+} } p_{j} -  \sum_{j \in J_{\ell h}\sp{-} } p_{j} 
 \phantom{ {} - p_{\ell}}
  + \sum_{j=\ell+1}\sp{n} a_{hj}p_{j} & \leq b_{h}
\qquad (h \in I_{\ell}\sp{0}). 
\label{Axbpl0}
\end{align}
With the use of 
$\alpha_{j} \leq p_{j} \leq \beta_{j}$ 
for $j =1,2, \ldots,\ell-1$,
we can eliminate 
$p_{1},\ldots,p_{\ell-1}$
from \eqref{Axbpl+},  \eqref{Axbpl-}, and \eqref{Axbpl0},
to obtain
\begin{align}
 \sum_{j \in J_{\ell i}\sp{+} } \alpha_{j} -  \sum_{j \in J_{\ell i}\sp{-} } \beta_{j} + p_{\ell}
  + \sum_{j=\ell+1}\sp{n} a_{ij}p_{j}  &\leq b_{i}
\qquad (i \in I_{\ell}\sp{+}) ,
\label{AxbBpell+2}
\\
 \sum_{j \in J_{\ell k}\sp{+} } \alpha_{j} -  \sum_{j \in J_{\ell k}\sp{-} } \beta_{j} 
 - p_{\ell}
  + \sum_{j=\ell+1}\sp{n} a_{kj}p_{j}  &\leq b_{k}
\qquad (k \in I_{\ell}\sp{-}), 
\label{AxbBpell-2}
\\
 \sum_{j \in J_{\ell h}\sp{+} } \alpha_{j} -  \sum_{j \in J_{\ell h}\sp{-} } \beta_{j} 
 \phantom{ {} - p_{\ell}}
  + \sum_{j=\ell+1}\sp{n} a_{hj}p_{j}  &\leq b_{h}
\qquad (h \in I_{\ell}\sp{0}) ,
\label{AxbBpell02}
\end{align}
while we have
\begin{align}
 p_{\ell} & \leq \beta_{\ell} ,
 \label{AxbBpell+1}
\\
 - p_{\ell} & \leq  -\alpha_{\ell} 
 \label{AxbBpell-1}
\end{align}
from \eqref{BoxIneq1}.  
It follows from 
\eqref{AxbBpell+2}, \eqref{AxbBpell-2},
\eqref{AxbBpell+1}, and \eqref{AxbBpell-1} 
that the interval for $p_{\ell}$ is given by
\begin{align}
& \max\{ \alpha_{\ell}, \  \max_{k \in I_{\ell}\sp{-}}
\{ -b_{k} + \sum_{j=\ell+1}\sp{n} a_{kj}p_{j}  
+ \sum_{j \in J_{\ell k}\sp{+} } \alpha_{j} -  \sum_{j \in J_{\ell k}\sp{-} } \beta_{j} 
\} \  \}
\nonumber \\
& \ \leq  p_{\ell}  \leq \ 
\min\{ \beta_{\ell}, \ 
\min_{i \in I_{\ell}\sp{+}}
\{  b_{i} - \sum_{j=\ell+1}\sp{n} a_{ij}p_{j} - \sum_{j \in J_{\ell i}\sp{+} } \alpha_{j}
   +  \sum_{j \in J_{\ell i}\sp{-} } \beta_{j} 
\} \  \} .
\label{AxbBpell}
\end{align}
Note that the inequality \eqref{AxbBpell} is solved for $p_{\ell}$
with the upper and lower bounds
depending on $p_{\ell+1}, \ldots, p_{n}$ 
and independent of $p_{1}, \ldots, p_{\ell-1}$.
It is emphasized that, 
for each $\ell$,
 the single inequality \eqref{AxbBpell}
is equivalent to 
the set of  inequalities consisting of 
\eqref{AxbBpell+2}, \eqref{AxbBpell-2},
\eqref{AxbBpell+1}, and \eqref{AxbBpell-1}.

We denote by $\mathrm{IQ}(\ell)$
the system of inequalities
consisting of
\eqref{AxbBpell+2}, \eqref{AxbBpell-2}, \eqref{AxbBpell02} for $\ell$, 
and 
\eqref{AxbBpell+1} and \eqref{AxbBpell-1}
for $\ell ,\ell+1,\ldots,n$, 
that is,
\begin{equation} \label{IQdef}
\mathrm{IQ}(\ell) = \{ 
\mbox{\eqref{AxbBpell+2}, \eqref{AxbBpell-2}, \eqref{AxbBpell02} for $\ell$}
\} \ \cup \ \{ 
\mbox{\eqref{AxbBpell+1}, \eqref{AxbBpell-1} for $\ell ,\ell+1,\ldots,n$} \}.
\end{equation}
Note that $\mathrm{IQ}(\ell)$
is a system of inequalities in variables 
$p_{\ell},p_{\ell+1},\ldots, p_{n}$,
and is free from $p_{1},\ldots, p_{\ell-1}$.
The number of inequalities in $\mathrm{IQ}(\ell)$
is equal to $|I|+ 2(n-\ell+1)$,
and $\mathrm{IQ}(1)$ is nothing but the combined system 
$A p \leq b, \ \alpha \leq x \leq \beta$.

It follows from the derivation above that, 
for $\ell = 1,2,\ldots,n$,
each inequality in $\mathrm{IQ}(\ell)$ is valid for $[P \cap B]_{\ell}$,
that is, every point in $[P \cap B]_{\ell}$
satisfies $\mathrm{IQ}(\ell)$.
The following theorem states that 
the converse is also true, that is,
$\mathrm{IQ}(\ell)$ gives a precise description
of $[P \cap B]_{\ell}$.
The proof of this theorem is given in Section \ref{SCfmPBproof}.

\begin{theorem}  \label{THelimICboxProj}
Let $f: \ZZ\sp{n} \to \ZZ \cup \{ +\infty \}$
be an integer-valued integrally convex function, $x \in \dom f$,
and $B$ be an integral box.
For each $\ell = 1,2,\ldots,n$,
the projection 
of $\subgR f(x) \cap B$
to the space of $(p_{\ell},p_{\ell+1},\ldots,p_n)$
is described by the system
$\mathrm{IQ}(\ell)$.
\finbox
\end{theorem}

\subsubsection*{Proof of Theorem~\ref{THICsubgrBox} by Theorem~\ref{THelimICboxProj}}

Since $\subgR f(x) \cap B$ is nonempty by assumption,
there exists a real vector $p$ satisfying 
the inequalities in 
\eqref{AxbBpell} for $\ell =1,2, \ldots,n$.
It is even true that, 
for any choice of 
$(p_{\ell},p_{\ell+1},\ldots,p_n)$
satisfying \eqref{AxbBpell} for $\ell, \ell+1, \ldots,n$,
the inequality of \eqref{AxbBpell}
for $\ell-1$
prescribes a nonempty interval for a possible choice of $p_{\ell-1}$.
Note that 
\eqref{AxbBpell02} for $\ell$
is an inequality to be counted as 
\eqref{AxbBpell+2} or \eqref{AxbBpell-2}
for some $\ell' \geq \ell +1$ or else
a trivial inequality between two constants.

As for integrality, the inequality of \eqref{AxbBpell} for $\ell=n$,
i.e.,
\begin{align*}
\max\{ \alpha_{n}, \  \max_{k \in I_{n}\sp{-}}
\{ -b_{k} + \sum_{j \in J_{nk}\sp{+} } \alpha_{j} -  \sum_{j \in J_{nk}\sp{-} } \beta_{j} 
\} \  \}
\leq
  p_{n}  \leq
\min\{ \beta_{n}, \ 
\min_{i \in I_{n}\sp{+}}
\{  b_{i} - \sum_{j \in J_{ni}\sp{+} } \alpha_{j}
   +  \sum_{j \in J_{ni}\sp{-} } \beta_{j} 
\} \  \},
\end{align*}
shows that we can choose an integral $p_{n} \in \ZZ$,
since 
$b_{i} \in \ZZ$ for all $i$, and
$\alpha_{j} \in \ZZ \cup \{ -\infty \}$ and
$\beta_{j} \in \ZZ \cup \{ +\infty \}$
for all $j$.
Then  the inequality of \eqref{AxbBpell} for $\ell=n-1$
shows that
we can choose an integral $p_{n-1} \in \ZZ$,
since its upper and lower bounds are both integers.
Continuing in this way we can see the existence of an integer vector $p \in \ZZ\sp{n}$
satisfying \eqref{AxbBpell} for all $\ell$.  
This shows $\subgR f(x) \cap B \cap \ZZ\sp{n} \neq \emptyset$,
which completes the proof of Theorem~\ref{THICsubgrBox}.

\begin{example} \rm  \label{EXicbox2dimFM}
We illustrate $\mathrm{IQ}(\ell)$ for the simple example in Example \ref{EXicbox2dim}.
By \eqref{BineqD2}
we have 
$(\alpha_{1}, \beta_{1}) = (2,+\infty)$ and
$(\alpha_{2}, \beta_{2}) = (-4,4)$.
 The inequalities in \eqref{ineqPicbox2dim}
for $P =  \subgR f(0,0)$ 
can be expressed as 
$A p \leq b$
with
\[
A = \left[
\begin{array}{rr}
   1 &    1  \\
   1 &   -1 \\
   -1 &   1 \\
   -1 &   -1 \\
   1 &    0 \\
   -1 &    0 \\
   0 &    1  \\
   0 &   -1  \\
\end{array} 
\right],
\qquad
b =
 \left[
\begin{array}{c}
   4   \\
   4  \\
   2  \\
   3  \\
   4  \\
   2  \\
   3  \\
   3  \\
\end{array} 
\right],
\]
where we assume that the row set of $A$ is indexed by $I=\{ r_{1}, r_{2},\ldots, r_{8} \}$
and the column set by $\{ 1,2 \}$.
First, for $\ell = 1$, we have
$I_{1}^{+} = \{ r_{1}, r_{2}, r_{5} \}$,
$I_{1}^{-} = \{ r_{3}, r_{4}, r_{6} \}$, and
$I_{1}^{0} = \{ r_{7}, r_{8} \}$.
Accordingly, \eqref{AxbBpell+2} and \eqref{AxbBpell-2} for $\ell = 1$ are given by
\begin{align*}
\mbox{\eqref{AxbBpell+2}$_{\ell = 1}$:} \qquad &
 p_{1} + p_{2} \leq 4, \quad
 p_{1} -p_{2} \leq 4, \quad
  p_{1} \leq 4, 
\\
\mbox{\eqref{AxbBpell-2}$_{\ell = 1}$:} \qquad &
 - p_{1} + p_{2}  \leq 2, \quad
 -p_{1} - p_{2} \leq 3, \quad
  -p_{1} \leq 2. 
\end{align*}
Then \eqref{AxbBpell} for $\ell = 1$ is given by
\begin{align*}
\mbox{\eqref{AxbBpell}$_{\ell = 1}$:} \qquad &
 \max\{ \alpha_{1}, -2 + p_{2}, -3-p_{2},-2 \} 
\leq p_{1}  \leq  
 \min \{ \beta_{1}, 4 - p_{2},   4+ p_{2}, 4   \} .
\end{align*}
The system
$\mathrm{IQ}(1)$ consists of all the inequalities in 
\eqref{ineqPicbox2dim} and \eqref{BineqD2}.

Next, for $\ell = 2$, we have
$I_{2}^{+} = \{ r_{1}, r_{3}, r_{7} \}$,
$I_{2}^{-} = \{ r_{2}, r_{4}, r_{8} \}$, and
$I_{2}^{0} = \{ r_{5}, r_{6} \}$.
Accordingly, 
\eqref{AxbBpell+2} 
and \eqref{AxbBpell-2} for $\ell = 2$ are given by
\begin{align*}
\mbox{\eqref{AxbBpell+2}$_{\ell = 2}$:} \qquad &
 \alpha_{1} + p_{2} \leq 4, \quad
 -\beta_{1} + p_{2} \leq 2, \quad
  p_{2} \leq 3, 
\\
\mbox{\eqref{AxbBpell-2}$_{\ell = 2}$:} \qquad &
 \alpha_{1} - p_{2} \leq 4, \quad
 -\beta_{1} - p_{2} \leq 3, \quad
  -p_{2} \leq 3 ,
\end{align*}
and  \eqref{AxbBpell} for $\ell = 2$ is given by
\begin{align*}
\mbox{\eqref{AxbBpell}$_{\ell = 2}$:} \qquad &
 \max\{ \alpha_{2}, -4 + \alpha_{1}, -3 -\beta_{1}, -3 \} 
\leq p_{2}  \leq  
\min \{ \beta_{2}, 4 - \alpha_{1},   2 + \beta_{1}, 3   \} .
\end{align*}
The system 
\eqref{AxbBpell02}
for $\ell = 2$ consists of two trivial inequalities:
\begin{align*}
\mbox{\eqref{AxbBpell02}$_{\ell = 2}$:} \qquad &
0 \leq 2, \quad 0 \leq 2+\infty.
\end{align*}
Then the system
$\mathrm{IQ}(2)$ consists of 
\eqref{AxbBpell+2}$_{\ell = 2}$,
\eqref{AxbBpell-2}$_{\ell = 2}$,
\eqref{AxbBpell02}$_{\ell = 2}$,
and
$-4 \leq p_{2}  \leq 4$
in \eqref{BineqD2};
note that the other inequality
$2 \leq p_{1}$ in \eqref{BineqD2} is not a member of 
$\mathrm{IQ}(2)$.

Using 
$(\alpha_{1}, \beta_{1}) = (2,+\infty)$
and
$(\alpha_{2}, \beta_{2}) = (-4,4)$, we obtain
\begin{align*}
\mbox{\eqref{AxbBpell}$_{\ell = 1}$:} \qquad &
 \max\{ 2, -2 + p_{2}, -3-p_{2},-2 \} 
\leq p_{1}  \leq  
 \min \{ 4 - p_{2},   4+ p_{2} ,4  \} ,
\\
\mbox{\eqref{AxbBpell}$_{\ell = 2}$:} \qquad &
 -2 \leq p_{2}  \leq  2,
\end{align*}
in agreement with Fig.~\ref{FGfmbox}.
It is noted, however, that 
$\mathrm{IQ}(2)$ contains redundant inequalities.
\finbox
\end{example}

\subsection{Proof of Theorem~\ref{THelimICboxProj}}
\label{SCfmPBproof}

In this section we give a proof of Theorem~\ref{THelimICboxProj},
which states that the projection $[P \cap B]_{\ell}$
is described by 
$\mathrm{IQ}(\ell)$ for $\ell=1,2,\ldots,n$,
where $P=\subgR f(x)$, $B=[\alpha,\beta]_{\RR}$,
and 
\begin{equation*} 
\mathrm{IQ}(\ell) = \{ 
\mbox{\eqref{AxbBpell+2}, \eqref{AxbBpell-2}, \eqref{AxbBpell02} for $\ell$}
\} \ \cup \ \{ 
\mbox{\eqref{AxbBpell+1}, \eqref{AxbBpell-1} for $\ell ,\ell+1,\ldots,n$} \}
\end{equation*}
defined in \eqref{IQdef} with
\begin{align*}
\mbox{\eqref{AxbBpell+2}:} \qquad &
 \sum_{j \in J_{\ell i}\sp{+} } \alpha_{j} -  \sum_{j \in J_{\ell i}\sp{-} } \beta_{j} 
  + \sum_{j=\ell+1}\sp{n} a_{ij}p_{j}
  + p_{\ell}
\leq b_{i}
\qquad (i \in I_{\ell}\sp{+}) ,
\\
\mbox{\eqref{AxbBpell-2}:} \qquad &
 \sum_{j \in J_{\ell k}\sp{+} } \alpha_{j} -  \sum_{j \in J_{\ell k}\sp{-} } \beta_{j} 
  + \sum_{j=\ell+1}\sp{n} a_{kj}p_{j} 
  - p_{\ell} \leq b_{k}
\qquad (k \in I_{\ell}\sp{-}) ,
\\
\mbox{\eqref{AxbBpell02}:} \qquad &
 \sum_{j \in J_{\ell h}\sp{+} } \alpha_{j} -  \sum_{j \in J_{\ell h}\sp{-} } \beta_{j} 
  + \sum_{j=\ell+1}\sp{n} a_{hj}p_{j}  
 \phantom{ {} - p_{\ell}}
  \leq b_{h}
\qquad (h \in I_{\ell}\sp{0}) ,
\\
\mbox{\eqref{AxbBpell+1}:} \qquad & \phantom{-}
 p_{\ell} \leq \beta_{\ell}, 
\\
\mbox{\eqref{AxbBpell-1}:} \qquad &
 {-p_{\ell}} \leq -\alpha_{\ell} .
\end{align*}
See \eqref{indIJelldef} for the definitions of
$I_{\ell}^{+}$, $I_{\ell}^{-}$, $I_{\ell}^{0}$,
$J_{\ell i}^{+}$, and $J_{\ell i}^{-}$.
We already know that 
each member of $\mathrm{IQ}(\ell)$ is a valid inequality for $[P \cap B]_{\ell}$,
and we need to prove that 
those inequalities are, in fact, sufficient for the description of $[P \cap B]_{\ell}$.

The subdifferential $P=\subgR f(x)$ is described by 
$A p \leq b$
in \eqref{ineqApb} and the integral box
$B$ by $\alpha \leq x \leq \beta$.
Hence $P \cap B$ is described by the
combined system 
$A p \leq b, \ \alpha \leq x \leq \beta$.
By applying the Fourier--Motzkin elimination procedure \cite{Sch86} to 
this combined system 
$A p \leq b, \ \alpha \leq x \leq \beta$,
we can obtain an inequality system to describe the projections
$[P \cap B]_{\ell}$ for $\ell = 1,2,\ldots,n$.
In the Fourier--Motzkin elimination,
the variable $p_{1}$ is eliminated first, and then $p_{2}, p_{3}, \ldots$,
to finally obtain an inequality in $p_{n}$ only.

It is worth while to reiterate here the technical subtlety
explained in Remark \ref{RMsystemPB}
concerning the relation between $[P \cap B]_{\ell}$ and $[P]_{\ell} \cap [B]_{\ell}$,
where $[P]_{\ell}$ and $[B]_{\ell}$
denote the projections of $P$ and $B$, respectively, 
to the space of $(p_{\ell},p_{\ell+1},\ldots,p_n)$.
By Theorem~\ref{THelimICProj}, the projection $[P]_{\ell}$ is
described by \eqref{FMp12n} for $p_{\ell},p_{\ell+1},\ldots,p_n$,
while $[B]_{\ell}$ is described obviously by
$\alpha_{j} \leq p_{j} \leq \beta_{j}$ $(j=\ell, \ell+1,\ldots,n)$.
If we had $[P \cap B]_{\ell} = [P]_{\ell} \cap [B]_{\ell}$,
we could describe $[P \cap B]_{\ell}$ by \eqref{FMconfunc}.
But this is not the case, as demonstrated  by Example \ref{EXicbox2dim}.

In the following we prove, by induction on $\ell=1,2,\ldots,n$,
that the projection $[P \cap B]_{\ell}$ 
to the space of $(p_{\ell},p_{\ell+1},\ldots,p_n)$
is described by the system $\mathrm{IQ}(\ell)$.
This statement is obviously true for $\ell = 1$,
since $\mathrm{IQ}(1)$ 
coincides with combined system 
$A p \leq b, \ \alpha \leq x \leq \beta$.
Fix $\ell$ with $1 \leq \ell \leq n-1$,
and assume that
$[P \cap B]_{\ell}$ is described by $\mathrm{IQ}(\ell)$.

In the following we shall prove that 
$[P \cap B]_{\ell+1}$ is described by $\mathrm{IQ}(\ell+1)$
by applying the Fourier--Motzkin elimination procedure 
to $\mathrm{IQ}(\ell)$ 
to eliminate the variable $p_{\ell}$.
The procedure consists of the following four types of elimination operations.

\begin{itemize}
\item
The addition of \eqref{AxbBpell+1} and \eqref{AxbBpell-1}
results in an obvious inequality
$\alpha_{\ell} \leq \beta_{\ell}$.

\item
The addition of \eqref{AxbBpell-2} for $k \in I_{\ell}\sp{-}$
and \eqref{AxbBpell+1}
results in 
\begin{align*}
\sum_{j \in J_{\ell k}\sp{+} } \alpha_{j} 
  -  \sum_{j \in J_{\ell k}\sp{-} } \beta_{j} - \beta_{\ell}
  + \sum_{j=\ell+1}\sp{n} a_{kj}p_{j}  \leq b_{k}.
\end{align*}
According to the value of 
$a_{k, \ell + 1} \in \{ +1, -1, 0 \}$,
this inequality is contained in
\eqref{AxbBpell+2},
\eqref{AxbBpell-2}, or
\eqref{AxbBpell02} for $\ell +1$,
and hence is contained in 
$\mathrm{IQ}(\ell+1)$.

\item
The addition of \eqref{AxbBpell+2} for $i \in I_{\ell}\sp{+}$
and
\eqref{AxbBpell-1}
results in 
\[
 \sum_{j \in J_{\ell i}\sp{+} } \alpha_{j} + \alpha_{\ell} -  \sum_{j \in J_{\ell i}\sp{-} } \beta_{j} 
  + \sum_{j=\ell+1}\sp{n} a_{ij}p_{j}  \leq b_{i} ,
\]
which is, similarly, a member of $\mathrm{IQ}(\ell+1)$.

\item
The addition of \eqref{AxbBpell+2} for $i \in I_{\ell}\sp{+}$
and
\eqref{AxbBpell-2} for $k \in I_{\ell}\sp{-}$
results in 
\begin{align}
\sum_{j \in J_{\ell i}\sp{+} } \alpha_{j} 
+ \sum_{j \in J_{\ell k}\sp{+} } \alpha_{j} 
- \sum_{j \in J_{\ell i}\sp{-} } \beta_{j}
- \sum_{j \in J_{\ell k}\sp{-} } \beta_{j}
+ \sum_{j=\ell+1}\sp{n} (a_{ij}+a_{kj})p_{j}
\leq b_{i} + b_{k} .
\label{AxbBpellfix+-2}
\end{align}
We prove later that this is 
a redundant inequality, implied by $\mathrm{IQ}(\ell+1)$.
\end{itemize}

Thus, all the inequalities generated
by the Fourier--Motzkin elimination procedure applied 
to eliminate the variable $p_{\ell}$
from $\mathrm{IQ}(\ell)$
are, in fact, implied by $\mathrm{IQ}(\ell+1)$.
On the other hand, 
$[P \cap B]_{\ell}$ is described by $\mathrm{IQ}(\ell)$
by the induction hypothesis,
and each inequality of $\mathrm{IQ}(\ell+1)$ 
is valid for $[P \cap B]_{\ell+1}$.
It then follows that
$[P \cap B]_{\ell+1}$ is described by $\mathrm{IQ}(\ell+1)$,
as desired.

\medskip

The rest of this section is devoted to the proof 
that the inequality \eqref{AxbBpellfix+-2}
is implied by $\mathrm{IQ}(\ell+1)$.
We may assume that 
$\alpha_{j}$ and $\beta_{j}$ appearing in \eqref{AxbBpellfix+-2}
are all finite-valued,
since otherwise this inequality is trivially true.
It turns out to be convenient to introduce
a factor of $1/2$ to \eqref{AxbBpellfix+-2}, to obtain
\begin{align}
\frac{1}{2}\left(
\sum_{j \in J_{\ell i}\sp{+} } \alpha_{j} 
+ \sum_{j \in J_{\ell k}\sp{+} } \alpha_{j} 
- \sum_{j \in J_{\ell i}\sp{-} } \beta_{j}
- \sum_{j \in J_{\ell k}\sp{-} } \beta_{j}
+ \sum_{j=\ell+1}\sp{n} (a_{ij}+a_{kj})p_{j}
\right)  
\leq \frac{1}{2}(b_{i} + b_{k}) .
\label{AxbBpellfix+-}
\end{align}
In the following we aim at proving  \eqref{AxbBpellfix+-} in place of \eqref{AxbBpellfix+-2}.

Consider an inequality
\begin{align}
\frac{1}{2}\left(
 \sum_{j \in J_{\ell i}\sp{+} } p_{j} 
 + \sum_{j \in J_{\ell k}\sp{+} } p_{j} 
- \sum_{j \in J_{\ell i}\sp{-} } p_{j}
- \sum_{j \in J_{\ell k}\sp{-} } p_{j}
 + \sum_{j=\ell+1}\sp{n} (a_{ij}+a_{kj})p_{j}
\right)  
\leq \frac{1}{2}(b_{i} + b_{k}) ,
\label{AxbBpellval+-}
\end{align}
which is obtained from \eqref{AxbBpellfix+-}
by replacing $\alpha_{j}$ and $\beta_{j}$ to $p_{j}$.
In this expression cancellations of the form of $+p_{j} - p_{j}$ occur
for $j \in (J_{\ell i}\sp{+} \cap J_{\ell k}\sp{-}) 
 \cup (J_{\ell k}\sp{+} \cap J_{\ell i}\sp{-})$.
On omitting these cancelling terms we obtain
\begin{align}
\frac{1}{2}\left(  
\sum_{j \in J_{\ell i}\sp{+}\setminus J_{\ell k}\sp{-}} p_{j} 
+ \sum_{j \in J_{\ell k}\sp{+} \setminus J_{\ell i}\sp{-}} p_{j} 
- \sum_{j \in J_{\ell i}\sp{-} \setminus J_{\ell k}\sp{+}} p_{j}
- \sum_{j \in J_{\ell k}\sp{-} \setminus J_{\ell i}\sp{+}} p_{j}
+ \sum_{j=\ell+1}\sp{n} (a_{ij}+a_{kj})p_{j} \right) 
\leq \frac{1}{2}(b_{i} + b_{k}) .
\label{AxbBpellval+-Can}
\end{align}
With the use of  coefficients 
$c= (c_{1},c_{2},\ldots,c_{n})$
defined by
\begin{align}
c_{j} = 
\begin{cases} 
+1  &  ( j \in J_{\ell i}\sp{+} \cap J_{\ell k}\sp{+}  ),
\\ 
-1  &  ( j \in J_{\ell i}\sp{-} \cap J_{\ell k}\sp{-}  ),
\\ 
+1/2  &  ( j \in J_{\ell i}\sp{+}\setminus (J_{\ell k}\sp{+} \cup  J_{\ell k}\sp{-}) 
          \ \ \mbox{or} \ \  
          j \in J_{\ell k}\sp{+} \setminus (J_{\ell i}\sp{+} \cup J_{\ell i}\sp{-}) ),
\\ 
-1/2  &  ( j \in J_{\ell i}\sp{-}\setminus (J_{\ell k}\sp{+} \cup  J_{\ell k}\sp{-}) 
          \ \ \mbox{or} \ \  
          j \in J_{\ell k}\sp{-} \setminus (J_{\ell i}\sp{+} \cup J_{\ell i}\sp{-}) ),
\\
(a_{ij}+a_{kj})/2 & (\ell + 1 \leq j \leq n),
\\
 0 & (\mbox{\rm otherwise}),
\end{cases}
\label{cjdef}
\end{align}
we can express \eqref{AxbBpellval+-Can} more compactly as
\begin{equation}
 c p  \leq \frac{1}{2}(b_{i} + b_{k}),
\label{AxbBpellval+-2} 
\end{equation}
where
$c p = \sum_{j=1}\sp{n} c_{j} p_{j}$.
By the definition \eqref{cjdef} we have
\begin{align}
& c_{j} \in \{-1, -\frac{1}{2}, 0, +\frac{1}{2}, +1\} \qquad (j=1 , \ldots , n),
\\ & 
 c_{j} = 0 \qquad  (j \in \{\ell\}  
 \cup (J_{\ell i}\sp{+} \cap J_{\ell k}\sp{-}) 
 \cup (J_{\ell k}\sp{+} \cap J_{\ell i}\sp{-}) ) .
\label{AxbBpellval+-3}
\end{align}

We observe that the inequality \eqref{AxbBpellval+-}
is, in fact, derived from  $A p \leq b$
 by adding 
\eqref{Axbpl+}
for $i \in I_{\ell}\sp{+}$
and 
\eqref{Axbpl-}
for $k \in I_{\ell}\sp{-}$
(and dividing by two).
By the definition of $A p \leq b$ in  \eqref{ineqApb},
we have
\[ 
 b_{i}=f(x + a_{i}) - f(x), \qquad
 b_{k}=f(x + a_{k}) - f(x),
\] 
as well as 
$c =  (a_{i} + a_{k})/2$,
where $a_{i}, a_{k} \in \{ -1, 0, +1\}\sp{n}$.
Recall that $I$ denotes the row set of the matrix $A$
and $a_{h}=(a_{h1},a_{h2},\ldots,a_{hn})$ is the $h$th row vector of $A$ for $h \in I$.

The following lemma depends heavily on the integral convexity of $f$.

\begin{lemma}  \label{LMcnvcombi}
There exist a subset $I' \ (\subseteq I)$
and positive weights
$\lambda_{h}$ indexed by $h \in I'$
(for convex combination)
such that
\begin{align}
 & 
\sum_{h \in I'} \lambda_{h}  = 1,\qquad \lambda_{h} > 0\quad(h \in I'), 
\label{AxbBpellval0} 
\\
& \sum_{h \in I'} \lambda_{h} a_{h} = c,
\label{AxbBpellval1} 
\\
& \sum_{h \in I'} \lambda_{h} b_{h} \leq \frac{1}{2}(b_{i} +b_{k}), 
\label{AxbBpellval2} 
\\
& c_{j} = 0  \ \Longrightarrow  \  a_{hj} = 0 \quad(h \in I'),
\label{AxbBpellval3} 
\\
& c_{j} > 0  \ \Longrightarrow \ a_{hj} \in \{0,+1\} \quad(h \in I'), 
\label{AxbBpellval4} 
\\
& c_{j} < 0 \ \Longrightarrow \  a_{hj} \in \{-1,0\} \quad(h \in I'). 
\label{AxbBpellval5}
\end{align}
\end{lemma}

\begin{proof}
We have
$b_{i}=f(x + a_{i}) - f(x)$, \ 
$b_{k}=f(x + a_{k}) - f(x)$, 
and
$c =  (a_{i} + a_{k})/2$.
By the integral convexity of $f$,
there exist 
$y\sp{(1)}, y\sp{(2)}, \ldots, y\sp{(m)} \in N(x+(a_{i}+a_{k})/2)$ 
such that
\begin{equation}\label{prfcnvcombi}
 \sum_{h=1}\sp{m} \lambda_{h} y\sp{(h)}
  = x + \frac{1}{2} (   a_{i} + a_{k} ), 
\qquad
 \sum_{h=1}\sp{m} \lambda_{h} f(y\sp{(h)})
\leq
 \frac{1}{2} (   f(x+a_{i}) +f(x+a_{k}) ) ,
\end{equation}
where $\lambda_h > 0$ for  $h =1,2,\ldots,m$ and $\sum_{h=1}\sp{m} \lambda_h = 1$.
Let $a_{h}$ be the row vector of $A$ that is equal to 
$y\sp{(h)} - x \in \{ -1, 0, +1\}\sp{n}$,
and $I'$ be the subset of $I$ corresponding to 
$y\sp{(h)} - x$ for  $h =1,2,\ldots,m$.
Then \eqref{prfcnvcombi} shows
\eqref{AxbBpellval1} and \eqref{AxbBpellval2}.
The last three conditions 
\eqref{AxbBpellval3}, \eqref{AxbBpellval4}, and \eqref{AxbBpellval5}
hold, since $a_{h}$ belongs to $N(c)$ for all $h \in I'$.
\end{proof}

Lemma~\ref{LMcnvcombi} enables us to show that the inequality 
\eqref{AxbBpellval+-Can}
(or \eqref{AxbBpellval+-2})
can be derived from the inequalities
corresponding to $I'$:
\begin{align}
 a_{h}p \leq b_{h} \qquad (h \in I').
\label{AxbBpIdash}
\end{align}
Indeed, by \eqref{AxbBpellval0}, \eqref{AxbBpellval1}, and \eqref{AxbBpellval2},
we obtain
\[
 c p  = 
 \left( \sum_{h \in I'} \lambda_{h} a_{h} \right) p 
 =   \sum_{h \in I'} \lambda_{h} ( a_{h}  p ) 
 \leq \sum_{h \in I'} \lambda_{h} b_{h}
 \leq  \frac{1}{2}(b_{i} +b_{k}).
\]
In the following, we use a variant of this argument
to show that another (related) inequality 
\begin{align}
\frac{1}{2}\left(  
\sum_{j \in J_{\ell i}\sp{+}\setminus J_{\ell k}\sp{-}} \alpha_{j} 
+ \sum_{j \in J_{\ell k}\sp{+} \setminus J_{\ell i}\sp{-}} \alpha_{j} 
- \sum_{j \in J_{\ell i}\sp{-} \setminus J_{\ell k}\sp{+}} \beta_{j}
- \sum_{j \in J_{\ell k}\sp{-} \setminus J_{\ell i}\sp{+}} \beta_{j}
+ \sum_{j=\ell+1}\sp{n} (a_{ij}+a_{kj})p_{j} \right) \leq \frac{1}{2}(b_{i} + b_{k})
\label{AxbBcl2}
\end{align}
can be derived from $\mathrm{IQ}(\ell+1)$.
Note that \eqref{AxbBcl2} resembles \eqref{AxbBpellval+-Can}.
Indeed, \eqref{AxbBcl2} is
obtained from \eqref{AxbBpellval+-Can}
by replacing $+p_{j}$ to $+\alpha_{j}$ and $-p_{j}$ to $-\beta_{j}$.

\begin{lemma}  \label{LMfromIQ(ell+1)}
The inequality \eqref{AxbBcl2} is implied by $\mathrm{IQ}(\ell+1)$.
\end{lemma}

\begin{proof}
Consider inequalities 
\begin{align}
 \sum_{j \in J_{\ell+1, h}^{+}} \alpha_{j} 
- \sum_{j \in J_{\ell+1, h}^{-}} \beta_{j} 
+ \sum_{j=\ell+1} a_{hj}p_{j} \leq b_{h}
\qquad (h \in I') .
\label{AxbBpellfix-JJ}
\end{align}
These inequalities 
are contained in $\mathrm{IQ}(\ell+1)$,
because 
each inequality in \eqref{AxbBpellfix-JJ}
is of the form of 
\eqref{AxbBpell+2},
\eqref{AxbBpell-2} or \eqref{AxbBpell02}
for $\ell+1$,
depending on $a_{h,\ell+1} \in \{ +1, -1, 0 \}$.

Since $a_{h \ell} = 0$ 
by \eqref{AxbBpellval+-3} and \eqref{AxbBpellval3},
we have
\begin{equation*}
 J_{\ell+1, h}^{+} =  \{ j \mid j<\ell,a_{hj}=+1 \},
\qquad 
 J_{\ell+1, h}^{-} =  \{ j \mid j<\ell,a_{hj}=-1 \} 
\end{equation*}
for all $h \in I'$.
Therefore, \eqref{AxbBpellfix-JJ} may be rewritten as 
\begin{align}
 \sum_{j<\ell,a_{hj}=+1} \alpha_{j} 
- \sum_{j<\ell,a_{hj}=-1} \beta_{j} 
+ \sum_{j=\ell+1} a_{hj}p_{j} \leq b_{h}
\qquad (h \in I')  .
\label{AxbBpellfix}
\end{align}
It should be clear that these inequalities are contained in $\mathrm{IQ}(\ell+1)$.

To show that the inequality \eqref{AxbBcl2}
is derived from \eqref{AxbBpellfix},
we form a convex combination of \eqref{AxbBpellfix}
using the weight $(\lambda_{h} \mid h \in I')$
in Lemma~\ref{LMcnvcombi}.
By the definition of $c_{j}$ in \eqref{cjdef}, we see that
\begin{align*}
\mbox{LHS of \eqref{AxbBcl2}}
 = 
 \sum_{j<\ell,c_{j}>0} c_{j} \alpha_{j} 
+ \sum_{j<\ell,c_{j}< 0} c_{j} \beta_{j} 
+ \sum_{j=\ell+1}\sp{n} c_{j} p_{j} .
\end{align*}
On substituting
\[
 c_{j}= \sum_{h \in I'} \lambda_{h} a_{hj} 
\qquad (j=1,\ldots,n)
\]
given in \eqref{AxbBpellval1},
we further obtain
\begin{align}
& 
\mbox{LHS of \eqref{AxbBcl2}}
\nonumber
\\ &
 = 
 \sum_{j<\ell,c_{j}>0}
 \left(  \sum_{h \in I'} \lambda_{h} a_{hj} \right)
\alpha_{j} 
+ \sum_{j<\ell,c_{j}< 0} 
 \left(  \sum_{h \in I'} \lambda_{h} a_{hj} \right)
 \beta_{j} 
+ \sum_{j=\ell+1}\sp{n} 
 \left(  \sum_{h \in I'} \lambda_{h} a_{hj} \right)
 p_{j} 
\nonumber
\\ &
 = \sum_{h \in I'} \lambda_{h} \left(
 \sum_{j<\ell,c_{j}>0} a_{hj} \alpha_{j} 
+ \sum_{j<\ell,c_{j}< 0} a_{hj} \beta_{j} 
+ \sum_{j=\ell+1}\sp{n} a_{hj} p_{j} 
\right)
\nonumber
\\ &
=
 \sum_{h \in I'} \lambda_{h} \left( 
 \sum_{j<\ell,a_{hj}=+1} \alpha_{j} 
 - \sum_{j<\ell, a_{hj}=-1} \beta_{j} 
 + \sum_{j=\ell+1} a_{hj}p_{j}
 \right) ,
\label{AxbBcl2L}
\end{align}
where
\eqref{AxbBpellval3}, \eqref{AxbBpellval4}, and \eqref{AxbBpellval5}
are used for the last equality.
On the other hand, the convex combination of \eqref{AxbBpellfix} shows
\begin{align}
 \sum_{h \in I'} \lambda_{h} \left( 
 \sum_{j<\ell,a_{hj}=+1} \alpha_{j} 
 - \sum_{j<\ell, a_{hj}=-1} \beta_{j} 
 + \sum_{j=\ell+1} a_{hj}p_{j}
 \right)
 \leq \sum_{h \in I'} \lambda_{h} b_{h} \leq \frac{1}{2}(b_{i} + b_{k}),
\label{AxbBcl2R}
\end{align}
where the second inequality is due to \eqref{AxbBpellval2}.
By combining  \eqref{AxbBcl2L} and \eqref{AxbBcl2R} we obtain \eqref{AxbBcl2}. 
\end{proof}

\medskip

Finally, we observe that,
while \eqref{AxbBcl2}
is implied by
$\mathrm{IQ}(\ell+1)$ by  Lemma \ref{LMfromIQ(ell+1)},
the inequality 
\eqref{AxbBpellfix+-}
in question
is obtained 
as the sum of 
\eqref{AxbBcl2}
and a trivial inequality
\begin{align}
 \frac{1}{2}
 \sum_{j \in   J_{\ell i}\sp{+} \cap J_{\ell k}\sp{-}} (\alpha_{j}-\beta_{j})
+  \frac{1}{2}
 \sum_{j \in  J_{\ell k}\sp{+} \cap J_{\ell i}\sp{-}} (\alpha_{j}-\beta_{j})
\leq 0 .
\label{alfbetaneg}
\end{align}
Therefore, 
\eqref{AxbBpellfix+-} is implied by, or redundant to, $\mathrm{IQ}(\ell+1)$.
This completes the proof of Theorem~\ref{THelimICboxProj}.

\bigskip

\noindent {\bf Acknowledgement}. 
The authors thank Satoru Fujishige for pointing out 
the connection to box convolution of bisubmodular functions 
expounded in Section \ref{SCbisubbox}, and
Akiyoshi Shioura for a helpful comment,
which led to Corollary~\ref{COsubgrICp}.
This work was supported by 
JSPS/MEXT KAKENHI JP20K11697, JP16K00023, and JP21H04979.



\begin{thebibliography}{99}

\bibitem{BL06}
Borwein, J.M., Lewis, A.S.:
Convex Analysis and Nonlinear Optimization: Theory and Examples, 2nd edn.
Springer, New York (2006)



\bibitem{CG91degseq}
Cunningham, W.H., Green-Kr{\'o}tki, J.:
$b$-matching degree-sequence polyhedra.
Combinatorica {\bf 11}, 219--230 (1991)




\bibitem{FT90}
Favati, P., Tardella, F.:
Convexity in nonlinear integer programming.
Ricerca Operativa {\bf 53}, 3--44 (1990)


\bibitem{FM19partII} 
Frank, F., Murota, K.: 
Discrete decreasing minimization, Part II:
Views from discrete convex analysis, 
arXiv: http://arxiv.org/abs/1808.08477  (2018).
Ver 4:  June 30, 2020.


\bibitem{FM20boxTDI} 
Frank, F., Murota, K.: 
A discrete convex min-max formula for box-TDI polyhedra.
Mathematics of Operations Research, 
published on-line (October 18, 2021)

https://doi.org/10.1287/moor.2021.1160



\bibitem{Fuj84fenc} 
Fujishige, S.:
Theory of submodular programs: 
A Fenchel-type min-max theorem and subgradients of submodular functions.
Mathematical Programming {\bf 29}, 142--155 (1984)


\bibitem{Fuj05book}
Fujishige, S.:
Submodular Functions and Optimization,
2nd edn.
Annals of Discrete Mathematics {\bf 58},
Elsevier, Amsterdam  (2005)


\bibitem{Fuj14bisubmdc} 
Fujishige, S.: 
Bisubmodular polyhedra, simplicial divisions, and discrete convexity.
Discrete Optimization  {\bf 12}, 115--120 (2014)



\bibitem{FP94}
Fujishige, S., Patkar, S.B.:
The box convolution and the Dilworth truncation of bisubmodular functions.
Report No. 94823-OR,
Forschungsinstitut f\"{u}r Diskrete Mathematik, Universit{\"a}t Bonn,
June 1994.



\bibitem{Haj85} 
Hajek, B.:
Extremal splittings of point processes.
Mathematics of Operations Research 
{\bf 10}, 543--556 (1985)




\bibitem{HL01}
Hiriart-Urruty, J.-B., Lemar{\'e}chal, C.:
Fundamentals of Convex Analysis.
Springer, Berlin (2001) 


\bibitem{IMT05} 
Iimura, T., Murota, K., Tamura, A.:
Discrete fixed point theorem reconsidered.
Journal of Mathematical Economics {\bf 41}, 1030--1036 (2005)



\bibitem{MM19projcnvl}
Moriguchi, S., Murota, K.:
Projection and convolution operations for integrally convex functions.
Discrete Applied Mathematics {\bf 255}, 283--298 (2019)




\bibitem{MMTT19proxIC} 
Moriguchi, S., Murota, K.,  Tamura, A.,    Tardella, F.:
Scaling, proximity, and optimization of integrally convex functions.
Mathematical Programming {\bf 175}, 119--154 (2019)


\bibitem{MMTT20dmc}
Moriguchi, S., Murota, K.,  Tamura, A.,    Tardella, F.:
Discrete midpoint convexity.
Mathematics of Operations Research {\bf 45}, 99--128 (2020)


\bibitem{Mdca98} 
Murota, K.:
Discrete convex analysis. 
Mathematical Programming  {\bf 83}, 313--371 (1998)


\bibitem{Mdcasiam} 
Murota, K.:
Discrete Convex Analysis.
Society for Industrial and Applied Mathematics, Philadelphia (2003)



\bibitem{Mdcaprimer07} 
Murota, K.:
Primer of Discrete Convex Analysis---Discrete 
versus Continuous Optimization (in Japanese).
Kyoritsu Publishing Co., Tokyo (2007)



\bibitem{Mbonn09} 
Murota, K.:
Recent developments in discrete convex analysis.
In: Cook, W., Lov{\'a}sz, L., Vygen, J. (eds.)
Research Trends in Combinatorial Optimization,
Chapter 11, pp.~219--260. Springer, Berlin (2009) 


\bibitem{Mdcaeco16} 
Murota, K.:
Discrete convex analysis: A tool for economics and game theory.
{Journal of Mechanism and Institution Design} 
{\bf 1}, 151--273 (2016)


\bibitem{Msurvop21} 
Murota, K.:
A survey of fundamental operations on discrete convex functions of various kinds.
Optimization Methods and Software {\bf 36}, 472--518 (2021)


\bibitem{Mopernet21} 
Murota, K.:
On basic operations related to network induction of discrete convex functions.
Optimization Methods and Software {\bf 36}, 519--559 (2021)



\bibitem{MS01rel} 
Murota, K., Shioura, A.:
Relationship of M-/L-convex functions with 
discrete convex functions by Miller and by Favati--Tardella.
Discrete Applied Mathematics {\bf 115}, 151--176 (2001)


\bibitem{MT20subgrIC} 
Murota, K., Tamura, A.:
Integrality of subgradients and biconjugates of integrally convex functions.
Optimization Letters
{\bf 14}, 195--208 (2020)



\bibitem{Qi88}
Qi, L.:
Directed submodularity, ditroids and directed submodular flows.
Mathematical Programming {\bf 42}, 579--599 (1988)



\bibitem{Roc70}
Rockafellar, R.T.:
Convex Analysis.
Princeton University Press, Princeton (1970) 


\bibitem{Sch86} 
Schrijver, A.:
Theory of Linear and Integer Programming.
 Wiley,  New York (1986)




\bibitem{Sch03}
Schrijver, A.:
Combinatorial Optimization---Polyhedra and Efficiency.
Springer, Heidelberg (2003)


\bibitem{TT21ddmc}
Tamura, A.,    Tsurumi, K.:
Directed discrete midpoint convexity.
Japan Journal of Industrial and Applied Mathematics {\bf 38}, 1--37 (2021)



\bibitem{Yan09fixpt} 
Yang, Z.:
Discrete fixed point analysis and its applications.
Journal of Fixed Point Theory and Applications {\bf 6}, 351--371 (2009)

\end{thebibliography}
\end{document}